\documentclass[11pt,a4papar]{article}

\usepackage{amsthm}
\usepackage{amssymb}
\usepackage{amscd}
\usepackage{amsmath}
\usepackage[all]{xy}
\usepackage[top=30truemm,bottom=30truemm,left=25truemm,right=25truemm]{geometry}
\usepackage{comment}
\usepackage{algorithmic,algorithm}
\usepackage{here}
\usepackage{graphicx}
\usepackage{color} %
\usepackage{enumerate} %
\usepackage{listliketab}
\usepackage{multirow}


\makeatletter
  
  \@addtoreset{algorithm}{subsection}
\makeatother

\theoremstyle{definition}

\theoremstyle{definition}

\theoremstyle{plain}
\newtheorem{theo}{Theorem}

\theoremstyle{plain}
\newtheorem{thm}{Theorem}[subsection]

\theoremstyle{definition}
\newtheorem{ex}[thm]{Example}

\theoremstyle{definition}

\theoremstyle{definition}

\theoremstyle{definition}
\newtheorem{defin}[thm]{Definition}

\theoremstyle{definition}
\newtheorem{rem}[thm]{Remark}

\theoremstyle{plain}
\newtheorem{prop}[thm]{Proposition}

\theoremstyle{plain}
\newtheorem{lem}[thm]{Lemma}

\theoremstyle{plain}

\theoremstyle{plain}
\newtheorem{cor}[thm]{Corollary}

\theoremstyle{definition}

\theoremstyle{definition}

\theoremstyle{definition}

\theoremstyle{definition}

\numberwithin{equation}{subsection}

\def\Z{\mathbb{Z}}

\def\dim{{\rm dim}}

\def\Im{{\rm Im}}
\def\Ker{{\rm Ker}}

\begin{document}

 \title{Computing representation matrices for the action\\ of Frobenius to cohomology groups}
\author{Momonari Kudo\thanks{Graduate School of Mathematics, Kyushu University. E-mail: \texttt{m-kudo@math.kyushu-u.ac.jp}}}
\maketitle

\begin{abstract}
This paper is concerned with the computation of representation matrices for the action of Frobenius to the cohomology groups of algebraic varieties.
Specifically we shall give an algorithm to compute the matrices for {\it arbitrary} algebraic varieties with defining equations over perfect fields of positive characteristic, and estimate its complexity.
Moreover, we propose a specific efficient method, which works for complete intersections.
\end{abstract}

\section{Introduction}\label{sec:Intro}

Let $K$ be a perfect field of positive characteristic $p$.
For a positive integer $r$, let $\mathbf{P}^r = \mathrm{Proj} ( S )$ denote the projective $r$-space, where
$S = K [ x_0, \ldots , x_r ]$ is the polynomial ring of $(r + 1)$ indeterminates over $K$.
Given a scheme $X \subset \mathbf{P}^r$ and $q \in \Z$, we denote by $\mathcal{O}_X$ and $H^q ( X, \mathcal{O}_X )$ its structure sheaf and its $q$-th cohomology group, respectively.
Let $F$ be the absolute Frobenius on $X$.
It is significant to compute the action of Frobenius to the cohomology groups, denoted by $F^{\ast,q}$ for the $q$-th cohomology group, since algebraic varieties defined over $K$ can be classified by investigating $F^{\ast,q}$.
For example, a non-singular curve is {\it superspecial} if and only if the Frobenius on the $1$st cohomology group is zero.
Specifically the representation matrix of $F^{\ast,1}$ for a curve is said to be its {\it Hasse-Witt matrix}.
Given a curve with explicit defining equations, computing $F^{\ast,1}$ allows us to determine its superspecialty.
Thus the computation of $F^{\ast,q}$ is one of the most crucial topics in computational algebraic geometry.

In cases of curves, i.e., one-dimensional algebraic varieties, there are many previous works for computing $F^{\ast,1}$, see e.g., \cite{BGS}, \cite{Gonz}, \cite {Harvey}, \cite{Matsuo}, \cite{Manin} and \cite{Yui}.
We here briefly review results on computation methods only for elliptic curves and hyperelliptic curves.

First, for the case of elliptic curves, there is a well-known and standard method to compute $F^{\ast,1}$.
Let $E$ be an elliptic curve in $\mathbf{P}^2$ defined by a cubic form $f \in K [ x, y, z]$, and $F$ the absolute Frobenius on $E$.
The $1$st cohomology group $H^1 ( E, \mathcal{O}_E )$ has a basis $\{ x^{-1} y^{-1} z^{-1} \}$, and thus the Frobenius action $F^{\ast,1} : H^1 ( E, \mathcal{O}_E ) \longrightarrow H^1 ( E, \mathcal{O}_E )$ sends the basis element to $f^{p-1} \cdot ( x^{-1} y^{-1} z^{-1} )^p$.
One can determine whether $F^{\ast,1} = 0$ or not by computing the values of coefficients in $f^{p-1}$, see \cite[Chapter I\hspace{-.1em}V]{Har} for more details. 

Second, for the case of hyperelliptic curves, in \cite{Manin} or in the proof of \cite[Proposition 2.1]{Yui}, one has an explicit method to get the representation matrix of $F^{\ast,1}$ when a hyperelliptic curve is given as an affine model $y^2 = f(x)$.
In this case, $F^{\ast,1}$ is determined by computing the values of coefficients in $f^{(p-1)/2}$, where $p$ is an odd prime.
Based on this method, several algorithms and their improvements to compute $F^{\ast,1}$ for hyperelliptic curves have been proposed, see e.g., \cite{BGS}, \cite {Harvey}, and \cite{Matsuo}.

While several algorithms for specific cases have been published as above, no general-purpose explicit algorithm, which works for {\it arbitrary} algebraic varieties (of dimension $\geq 1$) with defining equations, has been proposed and precisely analyzed yet.
(cf.\ In \cite[Section 5]{Kudo}, a basic framework is proposed, but neither written as a concrete algorithm nor precisely analyzed.)
This is because representing elements of $H^q (X, \mathcal{O}_X)$ depends on one's specific choice of an open covering for $X$.
In this paper, we give two algorithms for computing $F^{\ast,q}$ with $1 \leq q \leq r$ based on known theories in algebraic geometry and commutative algebra.
More concretely, we propose an algorithm (Algorithm (I)) for arbitrary varieties, and an algorithm (Algorithm (I\hspace{-.1em}I)) for complete intersections.
%
Our main theorem for Algorithm (I) is the following:

\begin{theo}\label{thm:main}
With notation as above, we fix $r$ the dimension of $\mathbf{P}^r$.
Given $1 \leq q \leq r$, characteristic $p$ and an algebraic variety $X \subset \mathbf{P}^r = \mathrm{Proj} (K[x_0, \ldots , x_r])$ with defining homogeneous polynomials $f_1, \ldots , f_t \in S = K [x_0, \ldots , x_r]$,
there exists an algorithm $(\mbox{\rm Algorithm (I)})$ for computing the rank of $F^{\ast,q} : H^q (X, \mathcal{O}_X) \longrightarrow H^q (X, \mathcal{O}_X)$ such that it terminates in 
\begin{equation}
O ( D^4 + \alpha^2 D^2 \mathrm{log} (p) ) \label{eq:mainThm}
\end{equation}
arithmetic operations over $K$,
under some assumptions.
Here $D$ is a certain invariant for $X$, and $\alpha$ depends on lifting morphisms between free resolutions for $S / \langle f_1, \ldots f_t \rangle$ and $S / \langle f_1^p, \ldots , f_t^p \rangle$.
\end{theo}

Algorithm (I) is based upon the local cohomology-based method to compute the sheaf cohomology (\cite{Eisen}, \cite{Kudo} and \cite{Smith}), and roughly divided into the following two steps:
({\it Step A}) Compute (graded minimal) free resolutions and lifting morphisms.
({\it Step B}) Compute the basis of $H^q (X, \mathcal{O}_X)$, and the representation matrix of $F^{\ast,q}$.
For computing free resolutions, we shall apply Kunz's theorem \cite{Kunz} in our algorithm, which reduces total time for the computation, and allows us a reasonable assumption to analyze its complexity.

For complete intersections, one can simplify Algorithm (I).
As stated in the following theorem, the simplified algorithm (Algorithm (I\hspace{-.1em}I)) has a more precise evaluation on the complexity:
\begin{theo}\label{thm:main2}
With notation as above, we fix $r$ the dimension of $\mathbf{P}^r$.
Let $S = K [x_0, \ldots , x_r]$, and $X = V(f_1, \ldots , f_t)$ a complete intersection in $\mathbf{P}^r$ with an $S$-regular sequence $(f_1, \ldots , f_t) \in S^t$.
Assume $d_{j_1 \ldots j_{t-1}} := \sum_{k=1}^{t-1} \mathrm{deg} ( f_{j_k} ) \leq r$ for all $1 \leq j_1 < \cdots < j_{t-1} \leq t$ and $\mathrm{gcd}( f_i, f_j ) = 1$ in $S$ for $i \neq j$. 
Given the characteristic $p$ and $(f_1, \ldots , f_t)$, there exists an algorithm $(\mbox{\rm Algorithm (I\hspace{-.1em}I)})$ for computing the rank of $F^{\ast,q} : H^q (X, \mathcal{O}_X) \longrightarrow H^q (X, \mathcal{O}_X)$ with $q=\mathrm{dim}(X)=r-t$ such that it terminates in 
\begin{equation}
O ( t M (p-1) ) \label{eq:mainThm2}
\end{equation}
arithmetic operations over $S$, where $M(m)$ denotes the cost for computing the $m$-th power in $S$.
\end{theo}
Since the strategy for Algorithm (I\hspace{-.1em}I) is theoretically the same as that for Algorithm (I), Algorithm (I\hspace{-.1em}I) can be viewed as a specific case of Algorithm (I).
By contrast, Algorithm (I\hspace{-.1em}I) has a computationally simplified and clarified structure, which makes the computation more efficient:
In practice Algorithm (I\hspace{-.1em}I) computes neither a free resolution nor a lifting homomorphism.
Specifically, for a given complete intersection $X=V( f_1, \ldots , f_t)$ defined by an $S$-regular sequence $(f_1, \ldots , f_t) \in S^t$ with certain conditions, we show that the rank of $F^{\ast, q} : H^q (X, \mathcal{O}_X) \longrightarrow H^q (X, \mathcal{O}_X)$ is completely determined by coefficients in $(f_1 \ldots f_t)^{p-1}$, where $q=\mathrm{dim}(X)=r-t$.
We prove this fact by constructing {\it Koszul complex of graded free modules} and certain special lifting homomorphisms.
Thanks to this fact, one also obtains a more precise evaluation on the complexity.
We also note that this simple method is viewed as a generalization of a standard method to compute $F^{\ast,q}$ for an elliptic curve in $\mathbf{P}^2$ (or a hypersurface in $\mathbf{P}^r$) with $q=\mathrm{dim}(X)=r-1$.

We demonstrate the correctness of our algorithms by computing some examples, one of which is $X_0(23)$, the (classical) modular curve of level $23$.
Our computational results for $X_0 (23)$ coincide with theoretical results computed by Yui's method \cite{Yui}.
We also examine efficiency of Algorithm (I).
As stated in Theorem \ref{thm:main}, under some assumptions, Algorithm (I) is proved to be performed in polynomial time with respect to $D$, $\alpha$ and $p$.
In particular, Algorithm (I) terminates in $O (\alpha^2 \mathrm{log}(p) )$ when $D$ is fixed.
For $X_0 (23)$ with $D=7$, we confirm, with our implementation over Magma \cite{Magma}, \cite{MagmaHP},
that Algorithm (I) performs in the estimated upper bound \eqref{eq:mainThm}.
With our algorithms and computational examples, one can compute representation matrices for the Frobenius action algorithmically for more general objects, which shall provide further theoretical/computational results.
(For instance, the author and Harashita already obtained theoretically new results on superspecial curves of genus $4$ by applying our algorithm for complete intersections, see \cite{KH16} and \cite{KH17-2}.)

The rest of this paper is organized as follows:
Section 2 is devoted to a brief review of a basic idea to compute representation matrices for the Frobenius action to the cohomology groups of algebraic varieties.
In addition, we review some known facts on the Frobenius functor for free modules, which are useful to make our (first) algorithm in Section 3 efficient and to evaluate its complexity.  
In Section 3, we prove Theorem \ref{thm:main}; we present an explicit algorithm to compute representation matrices for the action of Frobenius for general algebraic varieties, and estimate its complexity.
In Section 4, we prove Theorem \ref{thm:main2}; we give a new computational method for a complete intersection as a simplified version of our (first) algorithm proposed in Section 3.
Section 5 shows computational examples and experimental results obtained from our implementation over Magma.
In Section 6, we give some concluding remarks.

\section*{Acknowledgments}
The author thanks Allan Steel for helpful advice and discussion on computing lifting homomorphisms, and giving an efficient implementation of the division algorithm over free modules.
The author is also grateful Janko Boehm, Wolfram Decker, Shushi Harashita, and Gerhart Pfister for helpful advice and discussion on computing free resolutions for $S / I_p$.

\paragraph{Notation}
Given a graded module $M$ and an integer $\ell$, we denote by $M ( \ell )$ its $\ell$-twist given by $M (\ell)_t = M_{\ell + t}$.
In this paper, a homomorphism of graded $S$-modules means a graded homomorphism \textit{of degree zero} of graded $S$-modules.
Specifically we say that a graded homomorphism 
\[
\varphi : \bigoplus_{j = 1}^{t} S(-d_j) \longrightarrow \bigoplus_{j = 1}^{t^{\prime}} S(- d_j^{\prime}) \ ; \ {\bf v} \mapsto {\bf v} A
\]
is \textit{of degree zero} if each $( k, \ell )$-entry $g_{k, \ell} \in S$ of the representation matrix $A = \left( g_{k, \ell} \right)_{k, \ell}$ is homogeneous of degree $d_k - d_{\ell}^{\prime}$.

\section{Preliminaries}\label{sec:application}
In this section, we review a basic idea to compute representation matrices for the action of Frobenius to the cohomology groups of algebraic varieties.
The idea is given in \cite[Chapter 5]{Kudo}.

\subsection{Frobenius morphisms on schemes}
\begin{defin}[Absolute Frobenius]
Let $X$ be a scheme over $\mathbb{F}_q$, and $X \rightarrow \mathrm{Spec} (\mathbb{F}_q)$ its structure morphism.
The {\it absolute Frobenius} on $X$ is a morphism $F_{\mathrm{abs}} : X \rightarrow X$ with the identity map on $X$ and $a \mapsto a^p$ on sections.
\end{defin}

\begin{defin}[Relative Frobenius]\label{def:relative}
Let $X$ be a scheme over $\mathbb{F}_q$, and $X \rightarrow \mathrm{Spec} (\mathbb{F}_q)$ its structure morphism.
We define $X^{(p)}$ as the fiber product, and the {\it relative Frobenius} on $X$ as the induced morphism $F_X : X \rightarrow X^{(p)} $ in the following commutative diagram:
$$\xymatrix{
X \ar@(r,ul)[rrd]^{F_{\mathrm{abs}}} \ar[rd] \ar@(d,dr)[rdd] & & \\
 & X^{(p)} \ar[r] \ar[d] & X \ar[d] \\
& \mathrm{Spec} (\mathbb{F}_q) \ar[r]^{\cong} & \mathrm{Spec} (\mathbb{F}_q)
}$$
\end{defin}

\subsection{Our strategy to compute the Frobenius action}\label{sec_Frob}
Recall from Section \ref{sec:Intro} that $K$ is a perfect field of characteristic $p$.
Given homogeneous polynomials $f_1, \ldots , f_t \in S = K [ x_0, \ldots , x_r ]$, we denote by $V ( f_1, \ldots , f_t )$ the locus of the zeros in the projective $r$-space $\mathbf{P}^r = {\rm Proj} ( S )$ of the system of $f_1, \ldots , f_t$.
Put $X := V ( f_1, \ldots , f_t )$. 
Let $F:=F_{\mathrm{abs}}$ denote the absolute Frobenius on $X$.
Let $X^{(p)}$ be the scheme given in Definition \ref{def:relative}, and $F_X : X \rightarrow X^{(p)}$ the relative Frobenius.
The absolute (resp.\ relative) Frobenius induces the map $F^{\ast, i} : H^i (X, \mathcal{O}_X) \longrightarrow H^i (X, \mathcal{O}_X)$ (resp.\ $(F_X)^{\ast, i} : H^i (X^{(p)}, \mathcal{O}_{X^{(p)}}) \longrightarrow H^i (X, \mathcal{O}_X)$).
Computing $F^{\ast,i}$ or $(F_X)^{\ast, i}$ is important since it gives a tool for classifying algebraic varieties over $K$.
The absolute Frobenius $F : X \rightarrow X$ induces the Frobenius action $H^q ( X, {\cal O}_X ) \longrightarrow H^q ( X, {\cal O}_X )$, denoted by $F^{\ast, q}$.


The purpose of this paper is to present a method for computing $F^{\ast,q} : H^q ( X, {\cal O}_X ) \longrightarrow H^q ( X, {\cal O}_X )$ for $1 \leq q \leq r-1$ algorithmically.
Given a concrete $X$ and its covering, we can {\it theoretically} compute $H^i (X, \mathcal{O}_X)$ by \v{C}ech cohomology.
However, the representation of elements of $H^i (X, \mathcal{O}_X)$ depends on one's choice of an open covering $\mathcal{U} = \{ U_i \}_{i \in I}$ for $X$.
Moreover, one needs to represent the image of each basis element by the Frobenius as a linear combination of the same basis.
To compute such a representation algorithmically, we shall construct two morphisms, the composition of which coincides with $F$.
By this decomposition, we reduce the computation of $F^{\ast,q}$ into that of free resolutions and lifting morphisms of graded modules.
In what follows, we describe our strategy for computing the representation matrix for the Frobenius action $F^{\ast,q}$ to the $q$-th cohomology group $H^q (X, \mathcal{O}_X)$.
This framework has been proposed in \cite[Chapter 5]{Kudo} (but not neither written as a concrete algorithm nor precisely analyzed).

\paragraph{{\it The idea of our strategy:}}
We construct the following commutative diagram of morphisms:
$$\xymatrix{
& X_p \ar@(r,u)[rdd] & \\
X \ar[ru] \ar@(r,ul)[rrd]^{F_{\mathrm{abs}}} \ar[rd] \ar@(d,dr)[rdd] & & \\
 & X^{(p)} \ar[r] \ar[d] & X \ar[d] \\
& \mathrm{Spec} (\mathbb{F}_q) \ar[r]^{\cong} & \mathrm{Spec} (\mathbb{F}_q)
}$$
where we set $X_p:= V(f_1^p, \ldots , f_t^p)$, and $X \rightarrow X_p$ and $X_p \rightarrow X$ are certain morphisms defined below.
Instead of the composition of $X \rightarrow X^{(p)}$ and $X^{(p)} \rightarrow X$, we use that of $X \rightarrow X_p$ and $X_p \rightarrow X$ for computing the representation matrix for $F^{\ast, i}$.

Let $\mathcal{I}:= I^{\sim}$ and $\mathcal{I}_p :=(I_p)^{\sim}$ be the ideal sheaves associated to the homogeneous ideals $I := \langle f_1, \ldots , f_t \rangle_S$ and $I_p := \langle f_1^p, \ldots , f_t^p \rangle_S$, respectively.
Let $F_1$ denote the absolute Frobenius morphism on $\mathbf{P}^r$.
We then have the following commutative diagram:
$$\xymatrix{
& \qquad H^q (X, \mathcal{O}_X) \ar[ld] \ar[rr]^{\cong} \ar[d]^{\left(F_1 |_{X_{p}}\right)^{\ast,q}} \ar@(l,l)[dd]_{F^{\ast, q}} & & H^{q+1} ( \mathbf{P}^r, {\cal I} ) \ar[d]^{F_1^{\ast,q}}  \\
H^q (X^{(p)}, \mathcal{O}_{X^{(p)}}) \ar[rd] & \qquad H^q (X_p, \mathcal{O}_{X_p}) \ar[rr]^{\cong} \ar[d] & & H^{q+1} ( \mathbf{P}^r, {\cal I}_p ) \ar[d] \\
& \qquad H^q (X, \mathcal{O}_X) \ar[rr]^{\cong} & & H^{q+1} ( \mathbf{P}^r, {\cal I} )
}$$
where the morphism $H^q ( X_{p}, {\cal O}_{X_{p}} ) \longrightarrow H^q ( X, {\cal O}_X  )$ (resp. $H^{q+1} ( \mathbf{P}^r, {\cal I}_p ) \longrightarrow H^{q+1} ( \mathbf{P}^r, {\cal I} )$) is induced by the morphism
from ${\cal O}_{X_{p}}$ to ${\cal O}_X$ (resp. ${\cal I}_p$ to ${\cal I}$) corresponding to the homomorphism $S / I_p \rightarrow S / I$ (resp. $I_p \rightarrow I$).
Let $\psi$ denote the above $S/ I_p \rightarrow S/ I$, given by $h + I_p \mapsto h + I$.
The finitely generated graded $S$-module $S / I$ has a (minimal) free resolution of the form
$$\xymatrix{
0 \ar[r] & \bigoplus_{j=1}^{t_{r+1}} S \left( - d_j^{( r+1 )} \right) \ar[r]^(0.7){\varphi_{r+1}} &  \cdots \ar[r]^(0.3){\varphi_{2}} & \bigoplus_{j=1}^{t_{1}} S \left( - d_j^{( 1 )} \right) \ar[r]^(0.7){\varphi_1} & S \ar[r]^{\varphi_0} & S/ I  \ar[r]    & 0.           
}$$
We denote by $\left( g_{k,\ell}^{(i)} \right)_{k, \ell}$ the representation matrix of $\varphi_i$.
As we will see in the next subsection (Subsection \ref{sec_Frob_func}), the module $S / I_p$ has a graded free resolution of the form
$$\xymatrix{
0 \ar[r] & \bigoplus_{j=1}^{t_{r+1}} S \left( - d_j^{( r+1 )} p \right) \ar[r]^(0.7){\varphi_{r+1}^{(p)}} &  \cdots \ar[r]^(0.3){\varphi_{2}^{(p)}} & \bigoplus_{j=1}^{t_{1}} S \left( - d_j^{( 1 )} p \right) \ar[r]^(0.7){\varphi_1^{(p)}} & S \ar[r]^{\varphi_0^{(p)}} & S/ I_p  \ar[r]    & 0,           
}$$
where each $\varphi_{i}^{(p)}$ is given by the matrix with entries equal to the $p$-th powers of the entries of the matrix
for $\varphi_i$.
For the above free resolutions of $S/ I_p$ and $S/ I$, there exist $\psi_i$ for $0 \leq i \leq r+1$ such that the diagram
$$\xymatrix{
0 \ar[r] & \bigoplus_{j=1}^{t_{r+1}} S \left( - d_j^{( r+1 )} p \right)                 \ar[d]^{\psi_{r+1}} \ar[r]^(0.7){\varphi_{r+1}^{(p)}} &  \cdots \ar[r]^(0.3){\varphi_{2}^{(p)}} & \bigoplus_{j=1}^{t_{1}} S \left( - d_j^{( 1 )} p \right)          \ar[d]^{\psi_{1}} \ar[r]^(0.7){\varphi_1^{(p)}} & S \ar[r]^{\varphi_0^{(p)}} \ar[d]^{\psi_0} & S/ I_p    \ar[r]                  \ar[d]^{\psi}  & 0            \\
0 \ar[r] & \bigoplus_{j=1}^{t_{r+1}} S \left( - d_j^{( r+1 )} \right)                 \ar[r]^(0.7){\varphi_{r+1}} &  \cdots \ar[r]^(0.3){\varphi_{2}} & \bigoplus_{j=1}^{t_{1}} S \left( - d_j^{( 1 )} \right)         \ar[r]^(0.7){\varphi_1} & S    \ar[r]^{\varphi_0} & S/ I    \ar[r]    & 0            
}$$
commutes.
For a computation method of each morphism $\psi_i$, see e.g., \cite[Chapter 15]{Eisenbud_c}.
Let $C_i$ be the representation matrix of $\psi_i$.
Let $\Psi:= \psi^{\sim}$ denote the sheaf morphisms $\mathcal{O}_{\mathbf{P}^r} / \mathcal{I}_p \rightarrow \mathcal{O}_{\mathbf{P}^r} / \mathcal{I}$ induced by $\psi$.
Put $\Psi_i := \psi_i^{\sim}$,
\begin{eqnarray}
\mathcal{G}_{i} & := & \bigoplus_{j=1}^{t_{i}}\mathcal{O}_{\mathbf{P}^r}\left( - d_{j}^{( i )}\right), \quad 
\Phi_i := \varphi_i^{\sim} \mbox{ and} \nonumber \\
\mathcal{G}_{i}^{(p)} & := & \bigoplus_{j=1}^{t_{i}}\mathcal{O}_{\mathbf{P}^r}\left( - d_{j}^{( i )} p \right), \quad 
\Phi_i^{(p)} := \left( \varphi_i^{(p)} \right)^{\sim} \nonumber
\end{eqnarray}
for $0 \leq i \leq r + 1$, where $\varphi_i^{\sim}$, $(\varphi_i^{(p)})^{\sim}$ and $\psi_i^{\sim}$ denote the sheaf morphisms $\mathcal{G}_i \rightarrow \mathcal{G}_{i-1}$, $\mathcal{G}_i^{(p)} \rightarrow \mathcal{G}_{i-1}^{(p)}$ and $\mathcal{G}_i^{(p)} \rightarrow \mathcal{G}_i$ induced by $\varphi_i$, $\varphi_i^{(p)}$ and $\psi_i$, respectively.
The above commutative diagram of $S$-modules induces the following commutative diagram of the cohomology groups:
$$\xymatrix{
0 \ar[r] & H^r ( \mathbf{P}^r, \mathcal{G}_{r+1}^{(p)})                 \ar[d]^{H^r (\Psi_{r+1})} \ar[rr]^(0.6){H^r( \Phi_{r+1}^{(p)} ) } &  & \cdots \ar[rr]^(0.4){H^r (\Phi_{1}^{(p)}) } & & H^r (\mathbf{P}^r, \mathcal{G}_0^{(p)})          \ar[d]^{H^r (\Psi_{0})} \ar[rr]^{H^r (\Phi_0^{(p)})} & & H^r (\mathbf{P}^r, \mathcal{O}_{\mathbf{P}^r} / \mathcal{I}_p)                  \ar[d]^{H^r (\Psi)} \ar[r]    & 0            \\
0 \ar[r] & H^r ( \mathbf{P}^r, \mathcal{G}_{r+1})                 \ar[rr]^(0.6){H^r (\Phi_{r+1} )} &  & \cdots \ar[rr]^(0.4){H^r (\Phi_{1})} & & H^r (\mathbf{P}^r, \mathcal{G}_0)       \ar[rr]^{H^r(\Phi_0)} & & H^r (\mathbf{P}^r, \mathcal{O}_{\mathbf{P}^r} / \mathcal{I})     \ar[r]    & 0            
}$$
One can check that each horizontal sequence is a complex, since the functor $H^r (\cdot)$ is right exact.

\begin{lem}\label{lem:main}
With notation as above, we have the following commutative diagram:
$$\xymatrix{
\qquad H^q (X, \mathcal{O}_X) \ar[r]^{\cong} \ar[d]^{\left(F_1 |_{X_{p}}\right)^{\ast,q}} \ar@(l,l)[dd]_{F^{\ast, q}} & H^{q+1} ( \mathbf{P}^r, {\cal I} ) \ar[r]^(0.3){\cong}  \ar[d]^{F_1^{\ast,q}} & \Ker \left( H^r( \Phi_{r-q}) \right) / \Im \left( H^r ( \Phi_{r-q+1} ) \right) \ar[d]^{\mbox{\rm power } p}  \\
\qquad H^q (X_p, \mathcal{O}_{X_p}) \ar[r]^{\cong} \ar[d] & H^{q+1} ( \mathbf{P}^r, {\cal I}_p ) \ar[r]^(0.3){\cong} \ar[d] &  \Ker \left( H^r ( \Phi_{r-q}^{(p)} ) \right) / \Im \left( H^r (  \Phi_{r-q+1}^{(p)} ) \right) \ar[d]^{H^r (\Psi_{r-q})}\\
\qquad H^q (X, \mathcal{O}_X) \ar[r]^{\cong} & H^{q+1} ( \mathbf{P}^r, {\cal I} ) \ar[r]^(0.3){\cong} & \Ker \left( H^r \left( \Phi_{r-q} \right) \right) / \Im \left( H^r \left( \Phi_{r-q+1} \right) \right)
}$$
\end{lem}

\begin{proof}
See \cite[Section 5]{Kudo}.
\end{proof}
Here the representation matrix of $F^{\ast,q}$ is computed as follows.
\begin{enumerate}
	\item Compute a basis ${\cal B} = \{ \mathbf{b}_1, \ldots , \mathbf{b}_g \}$ of $H^q ( X, \mathcal{O}_{X} ) \cong \Ker \left( H^r( \Phi_{r-q}) \right) / \Im \left( H^r ( \Phi_{r-q+1} ) \right)$.
	\item Compute the vector, each entry of which is the $p$-th power of each entry of $\mathbf{b}_i$ for $1 \leq i \leq g$.
	We write the vector $\mathbf{b}_i^{(p)}$ for $1 \leq i \leq g$.
	\item Compute $\mathbf{b}_i^{(p)} \cdot {}^t C_{r-q}$ for $1 \leq i \leq g$, and the representation matrix of $F^{\ast,q}$ via the basis $\mathcal{B}$. 
\end{enumerate}

\begin{rem}\label{rem:basis}
\begin{enumerate}
	\item For integers $t$ and $m_j$ with $1 \leq j \leq t$, the $r$-th cohomology group
	\[
	H^r \left(\mathbf{P}^r, \bigoplus_{j=1}^t \mathcal{O}_{\mathbf{P}^r} (m_j) \right) \cong \bigoplus_{j=1}^t H^r (\mathbf{P}^r, \mathcal{O}_{\mathbf{P}^r} (m_j))
	\]
	is a finite-dimensional $K$-vector space with the basis
	\[
	\{ x_0^{\ell_0} \cdots x_r^{\ell_r} \mathbf{e}_j : 1 \leq j \leq t,\ (\ell_0, \ldots , \ell_r ) \in (\mathbb{Z}_{<0})^{r+1},\ \ell_0 + \cdots + \ell_r = -m_j \},
	\]
	where $\mathbf{e}_j$ denotes the vector with a $1$ in the $j$-th coordinate and $0$'s elsewhere.
	With this fact, the basis of $H^q ( X, \mathcal{O}_{X} ) \cong \Ker \left( H^r( \Phi_{r-q}) \right) / \Im \left( H^r ( \Phi_{r-q+1} ) \right)$ is given as follows:
	$H^r (\mathbf{P}^r, \mathcal{G}_{r-q})$ has a $K$-basis of the form
	\[
	\{ \mathbf{v}_1, \ldots , \mathbf{v}_{g^{\prime}} \} = \{ x_0^{\ell_0} \cdots x_r^{\ell_r} \mathbf{e}_j : 1 \leq j \leq t_{r-q},\ (\ell_0, \ldots , \ell_r ) \in (\mathbb{Z}_{<0})^{r+1},\ \ell_0 + \cdots  + \ell_r = -d_j^{(r-q)} \}.
	\]
	One has a $K$-basis of $\Ker \left( H^r( \Phi_{r-q}) \right) / \Im \left( H^r ( \Phi_{r-q+1} ) \right)$ as
	\[
	\left[
	\begin{array}{c}
	\mathbf{b}_1 \\
	\vdots \\
	\mathbf{b}_g 
	\end{array}
	\right]
	= B \cdot \left[ 
	\begin{array}{c}
	\mathbf{v}_1 \\
	\vdots \\
	\mathbf{v}_{g^{\prime}}
	\end{array}
	\right]
	\]
	for some $\mathrm{dim}_K H^q (X, \mathcal{O}_X) \times \mathrm{dim}_K H^r (\mathbf{P}^r, \mathcal{G}_{r-q})$ matrix $B$ over $K$.
	For the computation of $B$, see \cite[Section 3]{Kudo}.
	\item If the sequence $( f_1, \ldots , f_t)$ is $S$-regular, i.e., $X = V ( f_1, \ldots , f_t)$ is a complete intersection, then one can compute lifting maps between free $S$-modules in free resolutions for $S / I$ and $S / I_p$ via the \emph{Koszul complex}.
Moreover in such a case with $q = \mathrm{dim} (X)$, $F^{\ast,q}$ is determined by coefficients in $(f_1 \cdots f_t)^{p-1}$.
We will see that in Section \ref{sec:HWgeneral} (see also \cite[Appendix B]{KH16}).
	\item As we will see in the next subsection (Subsection \ref{sec_Frob_func}), one gets a free resolution for $S / I_p$ from a free resolution for $S / I$ with $p$-th power multiplications. 
\end{enumerate}
\end{rem}

\subsection{Frobenius functor for the category of modules}\label{sec_Frob_func} 

This subsection is devoted to properties of the Frobenius functor for the category of modules, see e.g., \cite{Mil} for more details.
Let $R$ be a ring of positive characteristic $p$, $M$ an $R$-module, and $f$ the Frobenius endomorphism on $R$.
We denote by ${}^f \! M$ the left $R$-module structure defined on $M$ by restriction of scalars via $f$, that is, for $r \in R$ and $m \in M$, we define $r \cdot m := r^p m$.

The Frobenius functor is defined as a functor from the category of $R$-modules to itself, and it is defined by
$F_R (M):= M \otimes_R {}^f \! R$.
Here, we enumerate in the following lemma some fundamental properties of the Frobenius functor $F_R (\cdot)$.

\begin{lem}\label{lem:Frob_func}
Let $R$ be a ring of positive characteristic $p$, $f$ the Frobenius endomorphism on $R$, and $F_R (\cdot)$ the Frobenius functor from the category of $R$-modules to itself.
\begin{enumerate}
	\item[$(1)$] $F_R ( \cdot )$ is right exact. 
	\item[$(2)$] $F_R ( R) = R \otimes_R {}^f \! R \simeq {}^f \! R \simeq R$ as $R$-modules via $a \otimes b \mapsto a \cdot b = a^p b$.
	For free modules, one has $F_R ( R^t ) = ( \bigoplus_{i=1}^t R ) \otimes_R {}^f \! R \simeq ( {}^f \! R )^t  \simeq R^t$ via $( a_1, \ldots , a_t ) \otimes b \mapsto ( a_1 \cdot b, \ldots , a_t \cdot b) = (a_1^p b, \ldots , a_t^p b)$.
	\item[$(3)$] For any ideal $J \subset R$, we have $F_R (R / J) = ( R / J ) \otimes_R {}^f \! R \simeq R / J_p$, where $J_p$ denotes the ideal generated by the $p$-th powers of elements of $J$.
	\item[$(4)$] Let $\varphi : R^t \rightarrow R^s$ be a homomorphism of $R$-modules, and $( r_{i,j} )_{i,j}$ a $t \times s$ matrix which represents $\varphi$ via standard bases.
	Then $F_R ( \varphi ) : R^t \rightarrow R^s$ is given by the matrix with entries equal to the $p$-th powers of the entries of the matrix $( r_{i,j} )_{i,j}$.
 \end{enumerate}
\end{lem}

\begin{proof}
(1) Since tensor product is right exact, the claim holds.
(2) Straightforward.
(3) The claim follows from $F_R (R / J) = ( R / J ) \otimes_R {}^f \! R \simeq J \cdot {}^f \! R$, where in this case $J \cdot {}^f \! R := \langle a^p x : a \in J,\ x \in {}^f \! R \rangle_R = J_p$.
(4) Let $\mathbf{e}_i$ be an element of the standard basis of $R^t$.
By (2), we identify $\mathbf{e}_i$ with $\mathbf{e}_i \otimes 1$, and it follows that
$F_R (\varphi) (\mathbf{e}_i) = ( \varphi \otimes \mathrm{id}_{{}^f \! R}) (\mathbf{e}_i \otimes 1) = ( \sum_{j=1}^s r_{i,j} \mathbf{e}_j ) \otimes 1 = \sum_{j=1}^s ( r_{i,j} \mathbf{e}_j \otimes 1) = \sum_{j=1}^s ( r_{i,j} \cdot 1 ) \mathbf{e}_j = \sum_{j=1}^s r_{i,j}^p \mathbf{e}_j$.
\end{proof}

\begin{thm}[Kunz's Theorem, \cite{Kunz}, Theorems 2.1 and 3.3]\label{thm:Kunz}
Let $R$ be a local ring of characteristic $p$.
Then $R$ is a regular ring if and only if $f^n$ is flat for all $n > 0$, where $f$ denotes the Frobenius endomorphism on $R$. 
\end{thm}

Since regularity and flatness can be each checked locally, we have the following corollary.

\begin{cor}\label{cor:Kunz}
Let $R$ be a ring of positive characteristic $p$.
Then $R$ is a regular ring if and only if $f^n$ is flat for all $n > 0$, where $f$ denotes the Frobenius endomorphism on $R$. 
\end{cor}

\begin{lem}\label{lem:Frob_resolution}
Let $L$ be a field of positive characteristic $p > 0$, and $R := L [ y_1, \ldots , y_n ]$ the polynomial ring with $n$ indeterminates over $L$.
Let $f_1, \ldots, f_t \in R$ be homogeneous polynomials with $d_j^{(1)} = \mathrm{deg} (f_j)$ for $1 \leq j \leq t$, and $J \subset R$ the ideal generated by $f_1, \ldots , f_t$.
Suppose that $R/ J$ has the following graded free resolution:
\begin{equation}
0 \rightarrow \bigoplus_{j=1}^{t_n}R\left( - d_{j}^{ (n) } \right) \stackrel{\varphi_{n}}{\to} \cdots \stackrel{\varphi_2}{\to} \bigoplus_{j=1}^{t_1}R\left( - d_{j}^{\left( 1 \right) } \right) \stackrel{\varphi_1}{\to} R \stackrel{\varphi_0}{\to} R / J \rightarrow 0. \label{eq:resolution}
\end{equation}
We denote by $\left( g_{k,\ell}^{(i)} \right)_{k, \ell}$ the representation matrix of $\varphi_i$.
Then $R / J_p$ has a graded free resolution of the form
\begin{equation}
0 \rightarrow \bigoplus_{j=1}^{t_n}R\left( - d_{j}^{( n )} p \right) \stackrel{\varphi_{n}^{(p)}}{\to} \cdots \stackrel{\varphi_2^{(p)}}{\to} \bigoplus_{j=1}^{t_1}R\left( - d_{j}^{\left( 1 \right)} p \right) \stackrel{\varphi_1^{(p)}}{\to} R \stackrel{\varphi_0^{(p)}}{\to} R / J_p \rightarrow 0, \label{eq:resolution_p} 
\end{equation}
where $J_p := \langle f_1^p, \ldots , f_t^p \rangle_R$, and
$\varphi_{i}^{(p)}$ is given by the matrix with entries equal to the $p$-th powers of the entries of the matrix
for $\varphi_i$ for each $0 \leq i \leq n$.
\end{lem}

\begin{proof}
By Lemma \ref{lem:Frob_func} and Corollary \ref{cor:Kunz} together with the fact that $R$ is a regular ring of dimension $n$, the sequence
\begin{equation}
0 \rightarrow \bigoplus_{j=1}^{t_n}R\left( - d_{j}^{( n )} p \right) \stackrel{\varphi_{n}^{(p)}}{\to} \cdots \stackrel{\varphi_2^{(p)}}{\to} \bigoplus_{j=1}^{t_1}R\left( - d_{j}^{\left( 1 \right)} p \right) \stackrel{\varphi_1^{(p)}}{\to} R \stackrel{\varphi_0^{(p)}}{\to} R / J_p \rightarrow 0 \label{seq:right_exact}
\end{equation}
is exact.
It is straightforward that the sequence \eqref{seq:right_exact} gives a graded free resolution for $R / J_p$.
\end{proof}

\begin{rem}
As in Section \ref{sec_Frob}, one computes both resolutions of $S / I$ and $S / I_p$ in order to compute representation matrices for the action of Frobenius to the cohomology groups.
Lemma \ref{lem:Frob_resolution} implies that it suffices for the computation to get only a resolution of $S / I$
(but in such a case one has to compute the $p$-th powers of the entries of each representation matrix of the resolution).
\end{rem}

\section{Proof of Theorem \ref{thm:main}: Algorithm and Complexity Analysis}\label{sec:alg_comp}

In this section, we prove Theorem \ref{thm:main} stated in Section \ref{sec:Intro};
we give a concrete algorithm to compute representation matrices for the action of Frobenius to the cohomology groups of algebraic varieties, prove its correctness, and determine its complexity.
As in the previous section, let $K$ be a perfect field of characteristic $p$, and let $S = K [ x_0, \ldots , x_r ]$ denote the polynomial ring of $r+1$ indeterminates over $K$.

\subsection{Concrete algorithm}\label{subsec:main_alg}

In the following, we give our main algorithm.
The idea of Algorithm (I) has been given in \cite[Section 5.3]{Kudo}.
Different from the computational method given in \cite[Section 5.3]{Kudo}, we clarify the structure of free resolutions and lifting morphisms for the input, and the images of the Frobenius on the cohomology groups.

We first fix $r$, the dimension of the projective space $\mathbf{P}^r$.
The inputs are a tuple of homogeneous polynomials $( f_1, \ldots , f_t ) \in S^t $ with $X:= V (f_1, \ldots , f_t) \subset \mathbf{P}^r$, the characteristic $p$, and an integer $1 \leq q \leq r-1$.
The output is the rank of the representation matrix of the Frobenius $F^{\ast,q} : H^q ( X, \mathcal{O}_X ) \longrightarrow H^q ( X, \mathcal{O}_X )$ for a suitable basis.

\paragraph{Algorithm (I)}\label{sec:main_alg}
Given homogeneous polynomials $f_1, \ldots , f_t \in S = K [ x_0, \ldots , x_r ]$, a rational prime $p$ and an integer $1 \leq q \leq r-1$, we give an algorithm to compute the representation matrix for the action of Frobenius to $H^q ( X, {\cal O}_X )$, where $X = V ( f_1, \ldots , f_t )$.
Our algorithm has the following two main procedures Steps A and B, and Step A (resp.\ Step B) consists of two (resp.\ three) sub-procedures (A-1)--(A-2) (resp.\ (B-1)--(B-3)):
\begin{enumerate}[Step A.]
\item We first compute free resolutions for $S / I$ and $S / I_p$, and lifting morphisms.
This step can be divided into the following two steps:
	\begin{enumerate}[(A-1)]
		\item Given homogeneous polynomials $f_1, \ldots , f_t \in S$ and a rational prime $p$, we compute (minimal) free resolutions for $S / I$ and $S / I_p$.
		Recall from Lemma \ref{lem:Frob_resolution} that we can compute a free resolution of $S / I_p$ by $p$-th power multiplications from the resolution of $S / I$.
		Specifically, we compute all the elements
		\begin{eqnarray}
		t_i,\ d_j^{( i )},\ \left( g_{k,\ell}^{(i)} \right)_{k,\ell},\ \mbox{and } \left( \left( g_{k,\ell}^{(i)} \right)^p \right)_{k,\ell} \mbox{ for } 1 \leq i \leq r+1  \label{eq:inputs1}
		\end{eqnarray}
		in Lemma \ref{lem:Frob_resolution}.
		These elements are determined from the resolution.
		\item Given the elements of \eqref{eq:inputs1}, we compute homomorphisms $\psi_i$ for $1 \leq i \leq r+1$ such that the following diagram commutes:
		$$\xymatrix{
		0 \ar[r] & \bigoplus_{j=1}^{t_{r+1}} S \left( - d_j^{( r+1 )} p \right)               \ar[d]^{\psi_{r+1}} \ar[r]^(0.7){\varphi_{r+1}^{(p)}} &  \cdots \ar[r]^(0.3){\varphi_{2}^{(p)}} & \bigoplus_{j=1}^{t_{1}} S \left( - d_j^{( 1 )} p \right)          \ar[d]^{\psi_{1}} \ar[r]^(0.7){\varphi_1^{(p)}} &  S          \ar[d]^{\psi_{0}} \ar[r]^{\varphi_0^{(p)}} & S/ I_p                  \ar[d]^{\psi} \ar[r]    & 0            \\
		0 \ar[r] & \bigoplus_{j=1}^{t_{r+1}} S \left( - d_j^{( r+1 )}\right)                 \ar[r]^(0.7){\varphi_{r+1}} &  \cdots \ar[r]^(0.3){\varphi_{2}} & \bigoplus_{j=1}^{t_{1}} S \left( - d_j^{\left( 1 \right)} \right)         \ar[r]^(0.7){\varphi_1} &  S  \ar[r]^{\varphi_0} & S/ I    \ar[r]    & 0            
		}$$
		where $\psi_0$ is the identity map on $S$, and $\psi$ is given by $h + I_p \mapsto h + I$.
		Specifically, we compute the representation matrices of $\psi_i$ via standard bases for $1 \leq i \leq r+1$, say
		\begin{eqnarray}
		C_i = \left( h_{k,\ell}^{(i)} \right)_{k,\ell} \mbox{ for } 1 \leq i \leq r+1. \label{eq:lifting}
		\end{eqnarray}
		Each homomorphism $\psi_i$ is called a {\it lifting homomorphism}.
		A method for computing lifting homomorphisms is given in e.g., \cite[Chapter 15]{Eisenbud_c}, which is based on division algorithms on free $S$-modules. 
	\end{enumerate}
\item Given the elements of \eqref{eq:inputs1} and \eqref{eq:lifting}, we next compute the representation matrix for the action of Frobenius $F^{\ast,q} : H^q ( X, \mathcal{O}_X ) \longrightarrow H^q ( X, \mathcal{O}_X )$. This step can be divided into the following three steps:
	\begin{enumerate}[(B-1)]
	\item Given the elements of \eqref{eq:inputs1}, we compute a basis of $H^q (X, \mathcal{O}_X)$.
	More concretely, we compute the basis via a certain isomorphism of $K$-vector spaces.
	As in Subsection \ref{sec_Frob}, put
	\begin{eqnarray}
		{\cal G}_{i} := \bigoplus_{j=1}^{t_{i}}{\cal O}_{\mathbf{P}^r}\left( - d_{j}^{( i )}\right), \ 
		\Phi_i := \varphi_i^{\sim}, \mbox{ for $0 \leq i \leq r + 1$}. \nonumber
	\end{eqnarray}
	As stated in Lemma \ref{lem:main}, we have the isomorphism
	\begin{eqnarray}
		H^q (X, \mathcal{O}_X) \cong \Ker \left( H^r ( \Phi_{r-q} ) \right) / \Im \left( H^r ( \Phi_{r-q+1} ) \right) \mbox{ for } 1 \leq q \leq r-1.
	\end{eqnarray}
	With this isomorphism, one can compute the basis $\mathcal{B} = \{ \mathbf{b}_1, \ldots , \mathbf{b}_g \}$ of $H^q (X, \mathcal{O}_X)$, and at the same time gets a $g \times g^{\prime}$ matrix $( b_{i,j} )_{i,j}$ such that
	\begin{eqnarray}
	\mathbf{b}_i = \sum_{j=1}^{g^{\prime}} b_{i,j} \mathbf{v}_{j}, \label{eq:repbasis}
	\end{eqnarray}
	where the set of $\mathbf{v}_j$'s is a basis of $H^r (\mathbf{P}^r, \mathcal{G}_{r-q})$, given in Remark \ref{rem:basis} (1).
	In other words, the matrix $( b_{i,j} )_{i,j}$ gives rise to a basis of the quotient vector space $\Ker \left( H^r ( \Phi_{r-q} ) \right) / \Im \left( H^r ( \Phi_{r-q+1} ) \right)$.
	For a more detailed description on computing the basis of $H^q (X, \mathcal{O}_X)$, see \cite{Kudo}.
	\item From the basis $\mathcal{B} = \{ \mathbf{b}_1, \ldots , \mathbf{b}_g \}$ of $H^q (X, \mathcal{O}_X)$, we compute the image of the basis by $F_1^{\ast, q}$, where $F_1$ denotes the absolute Frobenius on $\mathbf{P}^r$.
	More concretely, we compute the $p$-th power of each entry of $\mathbf{b}_i$ for $1 \leq i \leq g$.
	\item From the image $\mathbf{b}_i^{(p)} :=F_1^{\ast, q} (\mathbf{b}_i)$, the representation matrix $C_{r-q}$ of the lifting homomorphism $\psi_{r-q}$ and the basis $\mathcal{B}$ with the matrix $( b_{i,j} )_{i,j}$, we compute the representation matrix of $F^{\ast, q}$, and output its rank.
	More concretely, we compute $F^{\ast, q} (\mathbf{b}_i) = H^r (\psi_{r-q}^{\sim}) ( F_1^{\ast,q} (\mathbf{b}_i)) = H^r(\psi_{r-q}^{\sim}) (\mathbf{b}_i^{(p)}) = \mathbf{b}_i^{(p)} \cdot {}^t C_{r-q} $, and the representation matrix of $F^{\ast,q} : H^q ( X, \mathcal{O}_X ) \longrightarrow H^q ( X, \mathcal{O}_X)$ via the basis $\mathcal{B}$.
Finally, compute the rank and output it.  
	\end{enumerate}
\end{enumerate}


We here state the correctness of Algorithm (I).

\begin{prop}
$\mbox{\rm Algorithm (I)}$ outputs the rank of $F^{\ast,q}$.
\end{prop}

\begin{proof}
It is straightforward from the construction of Algorithm (I) that the output is the rank of the composition map of
$$\xymatrix{
\mathrm{Ker} (H^r (\Phi_{r-q})) / \mathrm{Im} (H^r (\Phi_{r-q+1})) \ar[rr]^{\mbox{ power }p} & & \mathrm{Ker} (H^r (\Phi_{r-q}^{(p)})) / \mathrm{Im} (H^r (\Phi_{r-q+1}^{(p)}))
}$$
and
$$\xymatrix{
\mathrm{Ker} (H^r (\Phi_{r-q}^{(p)})) / \mathrm{Im} (H^r (\Phi_{r-q+1}^{(p)}))  \ar[rr]^{H^r (\psi_{r-q}^{\sim})} & & \mathrm{Ker} (H^r (\Phi_{r-q})) / \mathrm{Im} (H^r (\Phi_{r-q+1})) .
}$$
By Lemma \ref{lem:main}, the rank of this composition coincides with that of $F^{\ast,q}$.
\end{proof}

\begin{rem}
In the case of $q=r$, one has
\begin{eqnarray}
	H^r (X, \mathcal{O}_X) \cong \mathrm{Coker} (H^r (\Phi_1))
\end{eqnarray}
and thus one can compute $F^{\ast,r}$ in a similar way to the cases of $1 \leq q \leq r-1$.
\end{rem}

In the following paragraph, we write down the computation of Steps (B-2) and (B-3) in algorithmic format, which shall give a useful information to implement the algorithm over mathematical softwares.
Computation of Step A is well-known in theory of Gr\"obner bases (see \cite[Chapter 6]{CLO2}, \cite[Chapter 15]{Eisenbud_c} and \cite[Subsections 4.1 and 4.2]{DL} for details). 
For Step (B-1), \cite{Kudo} gives a concrete description.
Let us omit to describe in this paper Steps A and (B-1).

\paragraph{A pseudocode for Steps (B-2) and (B-3)}

Recall from Step (B-1) of Algorithm (I) together with Remark \ref{rem:basis}, we get a basis $\{ \mathbf{b}_1, \ldots , \mathbf{b}_g \}$ for $H^q (X, \mathcal{O}_X)$ as follows:
The basis of $H^r (\mathbf{P}^r, \mathcal{G}_{r-q})$ computed in our algorithm is
	\[
	\{ \mathbf{v}_1, \ldots , \mathbf{v}_{g^{\prime}} \} = \{ x_0^{\ell_0} \cdots x_r^{\ell_r} \mathbf{e}_j : 1 \leq j \leq t_{r-q},\ (\ell_0, \ldots , \ell_r ) \in (\mathbb{Z}_{<0})^{r+1},\ \ell_0 + \cdots + \ell_r = -d_j^{(r-q)} \},
	\]
where $g^{\prime}:= \mathrm{dim}_K H^r (\mathbf{P}^r, \mathcal{G}_{r-q})$.
One has a $K$-basis of $\Ker \left( H^r( \Phi_{r-q}) \right) / \Im \left( H^r ( \Phi_{r-q+1} ) \right)$ as
	\[
	\left[
	\begin{array}{c}
	\mathbf{b}_1 \\
	\vdots \\
	\mathbf{b}_g 
	\end{array}
	\right]
	= B \cdot \left[ 
	\begin{array}{c}
	\mathbf{v}_1 \\
	\vdots \\
	\mathbf{v}_{g^{\prime}}
	\end{array}
	\right]
	\]
for some $g \times g^{\prime}$ matrix $B = (b_{i,j})_{i,j}$ over $K$.
Let $C_{r-q}$ denote the representation matrix for the lifting homomorphism $\psi_{r-q}$ computed in Step (A-2), say
	\[
	( \psi_{r-q} (\mathbf{e}_1), \ldots , \psi_{r-q} (\mathbf{e}_{t_{r-q}}) ) = ( \mathbf{e}_1, \ldots , \mathbf{e}_{t_{r-q}} ) \cdot C_{r-q}.
	\]
With these $B$, $\{ \mathbf{v}_1, \ldots , \mathbf{v}_{g^{\prime}} \}$ and $C_{r-q} = (h_{i,j})_{i,j}$ as inputs, we here write down a pseudocode (Algorithm \ref{alg:rep}) for Steps (B-2)-(B-3).

\begin{algorithm}[htb]
\caption{\texttt{RankOfFrobenius}($p, B, [ \mathbf{v}_i]_{i=1}^{g^{\prime}}, C_{r-q})$}\label{alg:rep}
\begin{algorithmic}[1]
\REQUIRE{characteristic $p$, a matrix $B$, a basis $[ \mathbf{v}_i]_{i=1}^{g^{\prime}}$, and a matrix $C_{r-q}$}
\ENSURE{the rank of $F^{\ast,q}$}
\STATE Compute the $p$-th power of each entry of $B$
\STATE $B^{(p)}$ $\leftarrow$ $( b_{i,j}^p )_{i,j}$
\FOR{$j=1$ \TO $g^{\prime}$}
	\STATE Write $\mathbf{v}_j = X_0^{\ell_0(j)} \cdots X_r^{\ell_r(j)} \mathbf{e}_{k(j)}$ for some $\ell_0 (j), \ldots , \ell_r (j) \in \mathbb{Z}_{<0}$ and $k (j) \geq 1$
	\STATE $\mathbf{v}_j^{(p)}$ $\leftarrow$ $X_0^{\ell_0(j) p} \cdots X_r^{\ell_r(j) p} \mathbf{e}_{k(j)}$
\ENDFOR
\STATE Compute
\[
B^{(p)} \cdot \left[ 
	\begin{array}{c}
	\mathbf{v}_1^{(p)} \\
	\vdots \\
	\mathbf{v}_{g^{\prime}}^{(p)}
	\end{array}
	\right]	
\cdot {}^t C_{r-q}
\]
\STATE Write
\[
B^{(p)} \cdot \left[ 
	\begin{array}{c}
	\mathbf{v}_1^{(p)} \\
	\vdots \\
	\mathbf{v}_{g^{\prime}}^{(p)}
	\end{array}
	\right]	
\cdot {}^t C_{r-q} = 
B^{\prime} \cdot \left[ 
	\begin{array}{c}
	\mathbf{v}_1 \\
	\vdots \\
	\mathbf{v}_{g^{\prime}}
	\end{array}
	\right]	
\]
for some $g \times g^{\prime}$ matrix over $K$
\STATE Solve the linear system $B X = B^{\prime}$
\STATE Compute the rank of $X$
\RETURN $\mathrm{rank}(X)$ 
\end{algorithmic}
\end{algorithm}

\begin{prop}\label{prop:subalg}
Algorithm \ref{alg:rep} outputs the rank of $F^{\ast,q}$.
\end{prop}

\begin{proof}
We first claim
\[
	\left[
	\begin{array}{c}
	F_1^{\ast, q} ( \mathbf{b}_1) \\
	\vdots \\
	F_1^{\ast, q} (\mathbf{b}_g) 
	\end{array}
	\right]
	= B^{(p)} \cdot \left[ 
	\begin{array}{c}
	\mathbf{v}_1^{(p)} \\
	\vdots \\
	\mathbf{v}_{g^{\prime}}^{(p)}
	\end{array}
	\right].
	\]
Indeed, we have
\begin{eqnarray}
F_1^{\ast, q} (\mathbf{b}_i) &=& F_1^{\ast, q} \left( \sum_{j=1}^{g^{\prime}} b_{i,j} \mathbf{v}_j \right) = F_1^{\ast, q} \left( \sum_{j=1}^{g^{\prime}} b_{i,j} X_0^{\ell_0 (j)}, \ldots , X_r^{\ell_r (j)} \mathbf{e}_{k(j)} \right) \nonumber \\
&=&  \sum_{j=1}^{g^{\prime}} (b_{i,j})^p  X_0^{\ell_0 (j) p}, \ldots , X_r^{\ell_r (j) p} \mathbf{e}_{k(j)}
 = \sum_{j=1}^{g^{\prime}} (b_{i,j})^p  \mathbf{v}_{j}^{(p)}, \nonumber 
\end{eqnarray}
where we write $\mathbf{v}_j = X_0^{\ell_0 (j)}, \ldots , X_r^{\ell_r (j)} \mathbf{e}_{k(j)}$ for some $\ell_0 (j), \ldots , \ell_r (j) \in \mathbb{Z}_{<0}$ and $k (j) \geq 1$.
Furthermore we have
\begin{eqnarray}
\left( \sum_{j=1}^{g^{\prime}} (b_{i,j})^p  \mathbf{v}_{j}^{(p)} \right) \cdot {}^t C_{r-q} & = & \left(  F_1^{\ast,q} (\mathbf{b}_i) \right) \cdot {}^t C_{r-q} = \left( H^r (\Psi_{r-q}) \circ F_1^{\ast,q} \right) (\mathbf{b}_i) \nonumber \\
& = & ( F^{\ast,q}) (\mathbf{b}_i) \in \Ker \left( H^r( \Phi_{r-q}) \right) / \Im \left( H^r ( \Phi_{r-q+1} ) \right). \nonumber 
\end{eqnarray}
Recall that $\Ker \left( H^r( \Phi_{r-q}) \right) \subset H^r (\mathbf{P}^r, \mathcal{G}_{r-q})$.
Hence there exist matrices $B^{\prime}$ and $X$ over $K$ such that
\[
B^{(p)} \cdot \left[ 
	\begin{array}{c}
	\mathbf{v}_1^{(p)} \\
	\vdots \\
	\mathbf{v}_{g^{\prime}}^{(p)}
	\end{array}
	\right]	
\cdot {}^t C_{r-q} = 
B^{\prime} \cdot \left[ 
	\begin{array}{c}
	\mathbf{v}_1 \\
	\vdots \\
	\mathbf{v}_{g^{\prime}}
	\end{array}
	\right]	
= X B \cdot \left[ 
	\begin{array}{c}
	\mathbf{v}_1 \\
	\vdots \\
	\mathbf{v}_{g^{\prime}}
	\end{array}
	\right],
\]
where ${}^t \! X$ gives the representation matrix of $F^{\ast,q}$ for the basis $\{ \mathbf{b}_1, \ldots , \mathbf{b}_g \}$.
\end{proof}

\begin{rem}\label{rem:free_resolution}
In Step A, we compute (minimal) free resolutions and lifting homomorphisms for the input modules.
Several computational methods for free resolutions and lifting homomorphisms have been proposed and implemented in computer algebra systems.
Such computations can be done in exponential time in general.
However, objects such as the cohomology groups and their related invariants should be determined by mathematical invariants of input structures.
From this, in our complexity analysis of the next subsection (Subsection \ref{subsec:Complexity}), we set such mathematical invariants obtained from the form of free resolutions and lifting homomorphisms as inputs.
\end{rem}

\subsection{Complexity analysis}\label{subsec:Complexity}

In this subsection, we investigate the complexity of Algorithm (I) given in the previous subsection.
Recall that the input objects of the algorithm are an integer $1 \leq q \leq r-1$, a rational prime $p$, and homogeneous polynomials $f_1, \ldots f_t \in S$.
Note that the computation of Step A is generally done in exponential time for the degrees and the number of the monomials of $f_i$'s.
For fixed $r$ and $q$, the output object is determined by $p$ and the elements of \eqref{eq:inputs1} and \eqref{eq:lifting}, which are computed in Step A.
From this, we estimate the complexity over $S$ of Algorithm (I) according to the parameters
$p$, $t^{(\mathrm{max})}:= \mathrm{max} \{ t_i : r- q - 1 \leq i \leq r - q + 1 \}$,
$d^{(\mathrm{max})} := \mathrm{max} \{ d^{(i, \mathrm{max})} : r - q - 1 \leq i \leq r - q + 1 \}$,
where $d^{(i, \mathrm{max})}:= \mathrm{max} \{ d_j^{(i)} : 1 \leq j \leq t_i \}$.

\begin{prop}\label{prop_comp}
With notations as above,
Step B of $\mbox{\rm Algorithm (I)}$ in Subsection \ref{subsec:main_alg} (not counting the computation of a basis of $H^r ( \mathbf{P}^r, {\cal G}_i )$ for $r-q-1 \leq i \leq r-q+1$) terminates in
\begin{eqnarray}
O \left( \left( t^{( {\rm max})} ( d^{( \max )} )^r \right)^4  + \left( ( t^{(\mathrm{max})} ) ( d^{(\mathrm{max})})^r \right)^{2 } \mathrm{log}(p) \right) \label{result_comp}
\end{eqnarray}
arithmetic operations over $S = K [ X_0 , \ldots , X_r ]$.
\end{prop}

\begin{proof}

\begin{enumerate}[(B-1)]
\item[(B-1)] {\it The bases of $H^q ( X, \mathcal{O}_X )$}:
	From \cite[Proposition 3.5.1]{Kudo}, this computation is done in
	\begin{eqnarray}
	O \left( \left( t^{( {\rm max})} ( d^{( \max )} )^r \right)^4 \right) \label{comp_step1}
	\end{eqnarray}
	arithmetic operations over $S$.
	Recall from $\mbox{\rm Algorithm (I)}$ in Subsection \ref{sec:main_alg} that this step outputs a basis $\mathcal{B}$ of $H^q (X, \mathcal{O}_X) \cong \Ker \left( H^r ( \Phi_{r-q} ) \right) / \Im \left( H^r ( \Phi_{r-q+1} ) \right)$ with a basis $\mathcal{V} = \{ \mathbf{v}_1, \ldots , \mathbf{v}_{g^{\prime}} \}$ of $H^r (\mathbf{P}^r, \mathcal{G}_{r-q})$ and a $g \times g^{\prime}$ matrix $( b_{i,j})_{i,j}$ such that \eqref{eq:repbasis} is satisfied.
\item[(B-2)] {\it The image of $\mathcal{B}$ by $F_1^{\ast,q}$}:
	In this step, for each basis element $\mathbf{b}_i \in \mathcal{B}$, we compute the $p$-th power of its each entry.
	Recall that $\mathbf{b}_i$ is of the form $\sum_{j=1}^{g^{\prime}} b_{j,i} \mathbf{v}_j$.
	As stated in the proof of Proposition \ref{prop:subalg}, we have
	\begin{eqnarray}
	\mathbf{b}_i^{(p)}:= F_1^{\ast, q} (\mathbf{b}_i) = F_1^{\ast} \left( \sum_{j=1}^{g^{\prime}} b_{i,j} \mathbf{v}_j \right) =  \sum_{j=1}^{g^{\prime}} (b_{i,j})^p \mathbf{v}_j^{(p)}, \nonumber 
	\end{eqnarray}
	where $\mathbf{v}_j^{(p)}$ denotes the vector with entries equal to the $p$-th powers of the entries of $\mathbf{v}_j$.
	Assume here that one computes exponentiation in $O (\mathrm{log}(e))$ arithmetic operations, where $e$ is the exponent.
	Then it follows from $g^{\prime} = O (t^{(\mathrm{max})} (d^{(\mathrm{max})})^r )$ that computing $\mathbf{b}_i^{(p)}$ is done in
	\[
	O \left( t^{(\mathrm{max})} (d^{(\mathrm{max})})^r \mathrm{log} (p) \right).
	\]	
	Since $g = \mathrm{dim}_K H^q (X, \mathcal{O}_X ) = O (t^{(\mathrm{max})} ( d^{(\mathrm{max})} )^r )$, this computation terminates in
	\begin{eqnarray}
 	O \left( \left( t^{(\mathrm{max})} ( d^{(\mathrm{max})})^r \right)^2 \mathrm{log}(p) \right) \label{comp_step2}
 	\end{eqnarray}
	arithmetic operations over $K$.
\item[(B-3)] {\it The representation matrix of $F^{\ast}$}:
	In this part, we first compute the image of $\mathcal{B}^{(p)}:= F_1^{\ast,q} (\mathcal{B} )$ by $H^q (\Psi_{r-q})$.
	Recall that $g= \# \mathcal{B}^{(p)} = O  (t^{(\mathrm{max})} ( d^{(\mathrm{max})} )^r )$.
	Since $C_i$ is a $(t_i \times t_i)$-matrix over $S$, the computation runs in
	\begin{eqnarray}
	O \left( ( t^{(\mathrm{max})} )^3 ( d^{(\mathrm{max})} )^r \right) \label{comp_step3-1}
	\end{eqnarray}
	arithmetic operations over $S$.
	For each $1 \leq i \leq g$, we compute $x_{i,1}, \ldots x_{i,g} \in K$ such that
	\begin{eqnarray}
	\mathbf{b}_i^{(p)} \cdot {}^t C_{r-q} = \sum_{j=1}^g x_{i,j} \mathbf{b}_j.
	\end{eqnarray}
	To find $x_{i,1}, \ldots x_{i,g} \in K$, we solve a linear system.
	More concretely, we first represent $\mathbf{b}_i^{(p)} \cdot C_{r-q}$ as
	\begin{eqnarray}
		C_{r-q} \mathbf{b}_i^{(p)} = \sum_{j=1}^{g^{\prime}} b_{i,j}^{\prime} \mathbf{v}_j
	\end{eqnarray}
	for each $i$, where $\mathcal{V} = \{ \mathbf{v}_1, \ldots \mathbf{v}_{g^{\prime}} \}$ is a basis of $H^r (\mathbf{P}^r, \mathcal{G}_{r-q})$.
	Then we solve
	\begin{eqnarray}
	( x_{i,1}, \ldots , x_{i,g} ) \left( \begin{array}{ccc}
	b_{1,1} & \cdots & b_{1,g^{\prime}} \\
	\vdots & & \vdots \\
	b_{g,1} & \cdots & b_{g,g^{\prime}}
	\end{array} \right)
	= ( b_{i,1}^{\prime}, \ldots b_{i,g^{\prime}}^{\prime} )
	\end{eqnarray}
	for each $1 \leq i \leq g$.
	The size of the coefficient matrix $( b_{i,j} )_{i,j}$ is $g^{\prime} = O ( t^{(\mathrm{max})} (d^{(\mathrm{max})})^r )$.
	Thus the computation is done in
	\begin{eqnarray}
	O \left( ( t^{(\mathrm{max})} )^4 ( d^{(\mathrm{max})} )^{4 r} \right) \label{comp_step3}
	\end{eqnarray}
	arithmetic operations over $K \subset S$.
	Hence the complexity in this step is estimated as
	\begin{eqnarray}
	O \left( \left( t^{(\mathrm{max})} ( d^{(\mathrm{max})} )^{r} \right)^4 \right) \label{comp_step3-2}
	\end{eqnarray}
	arithmetic operations over $K \subset S$.
\end{enumerate}
By putting all the steps together, i.e., considering \eqref{comp_step1}-\eqref{comp_step3-2}, Proposition \ref{prop_comp} holds.
\end{proof}

By Proposition \ref{prop_comp} together with \cite[Corollary 3.5.2]{Kudo}, we can determine the complexity of Step B of Algorithm (I) in Subsection \ref{subsec:main_alg} over the ground field $K$.

\begin{cor}\label{rem_comp1}
The notations are same as in Proposition \ref{prop_comp}.
Let $\alpha$ be the maximum of the number of the terms of the entries of $C_{r-q}$ and $A_i$ for $r - q \leq i \leq r - q + 1$.
The arithmetic complexity of Step B of $\mbox{\rm Algorithm (I)}$ over $K$ (not counting the computation of bases of $H^r ( \mathbf{P}^r, {\cal G}_i )$ for $r - q - 1 \leq i \leq r - q + 1$) is
\begin{eqnarray}
O \left( \left( t^{( {\rm max} )} ( d^{( \max )} )^r \right)^4 + \alpha^2 \left( t^{( {\rm max} )} ( d^{( \max )} )^r \right)^2 \mathrm{log} (p) \right) \label{result_comp_2}.
\end{eqnarray}
\end{cor}
The value
\begin{eqnarray}
D := \max \{ \dim_K H^r ( \mathbf{P}^r, {\cal G}_i ) : r-q - 1 \leq i \leq r-q + 1 \} \label{def:D}
\end{eqnarray}
is also appropriate as an asymptotic parameter to estimate the complexity of Step B in Algorithm (I).
We describe in the following the reason why the parameter $D$ is suitable.
Recall that we have
\begin{equation}
\dim_K H^r ( {\mathbf{P}^r, \cal G}_i ) = \sum_{j = 1}^{t_i} \binom{ d_j^{( i )} }{r} \mbox{ for } r-q - 1 \leq i \leq r-q + 1 \label{eq:sum_dim}
\end{equation}
and thus $\dim_K H^r ( \mathbf{P}^r, {\cal G}_i ) = O \left( t^{( {\rm max} )} ( d^{( \max )} )^r \right)$.
The values $t_i$ and $d_j^{( i )}$ are uniquely determined for the input module $S / I$ since the form of the minimal resolution of $S / I$ is uniquely determined up to isomorphism of minimal resolutions.
Thus each value \eqref{eq:sum_dim} is also uniquely determined by $S / I$.
From this, we can take $D$ as an asymptotic parameter to estimate Step B of Algorithm (I).
In a similar way to Corollary \ref{rem_comp1}, the arithmetic complexity of Step B of Algorithm (I) with respect to $D$ over $K$ is estimated as follows.

\begin{cor}\label{rem_comp2}
The notations are same as in Proposition \ref{prop_comp} and Corollary \ref{rem_comp1}.
We fix $r$ and set $D := \max \{ \dim_K H^r ( \mathbf{P}^r, {\cal G}_i ) \ ; \ r-q - 1 \leq i \leq r-q + 1 \}$ as in (\ref{def:D}).
Then the arithmetic complexity of Step B of $\mbox{\rm Algorithm (I)}$ in Subsection \ref{subsec:main_alg} over $K$ is
\begin{eqnarray}
O ( D^4 + \alpha^2 D^2 \mathrm{log}(p) ) \label{result_comp_3},
\end{eqnarray}
where $\alpha$ is same as in Corollary \ref{rem_comp1}.
\end{cor}

In addition, we can give the binary complexity of Step B of Algorithm (I) for $K = \mathbb{F}_p$, where $p$ is a rational prime.

\begin{cor}\label{rem_comp3}
The notations are same as in Proposition \ref{prop_comp}, Corollary \ref{rem_comp1} and Corollary \ref{rem_comp2}.
Let $p$ be a rational prime and $K$ the finite field of $p$ elements, say $K = \mathbb{F}_p$.
We fix $r$, and assume that the computation in $\mathbb{F}_p$ is done in $O ( (\mathrm{log} (p))^3 )$ bit operations.
Then the binary complexity of Step B of $\mbox{\rm Algorithm (I)}$ in Subsection \ref{subsec:main_alg} is
\begin{eqnarray}
O ( D^4 (\mathrm{log} (p))^3 + \alpha^2 D^2 (\mathrm{log} (p) )^4 ) \label{result_comp_4},
\end{eqnarray}
where $\alpha$ is same as in Corollary \ref{rem_comp1}.
\end{cor}

\subsection{Comparison with conventional computations over affine hypersurfaces}
We briefly compare our Algorithm (I) in Section \ref{sec:main_alg} with conventional computations over affine hypersurfaces, specifically hyperelliptic curves in $\mathbf{P}^2$.
In this case, the input variety of our algorithm is given as a {\it projective} model defined by homogeneous polynomials in $S= K[x_0, \ldots , x_r]$ for some $r$ whereas that of the conventional algorithms in \cite{BGS}, \cite {Harvey}, \cite{Matsuo}, \cite{Manin} and \cite{Yui} is given as an {\it affine} model of the form $y^2 - g(x)=0$ defined by {\it one} polynomial $y^2 - g(x) \in K[x, y]$ not necessary to be homogeneous.
For a comparison, we here assume that the input variety of our algorithm is given as the locus of the zeros of one homogeneous polynomial $f \in K[X,Y,Z]$, and that its de-homogenization is of the form $y^2-g(x)$ for some $g \in K[x,y]$.
Let $X = V(f)$ be the hypersurface in $\mathbf{P}^2$ defined by $f=0$.
It is straightforward that the output of our algorithm with an input $f$ coincides with that of the conventional algorithms with an input $y^2 = g(x)$.
Thus one can choose one of the conventional algorithms with the input $y^2 = g(x)$ to compute the Frobenius on $H^1(X, \mathcal{O}_X)$.
In this sense, our algorithm is viewed as a generalization of the conventional computations.
Moreover we interpret that the complexity of our algorithm is equivalent to that of the algorithms in \cite{BGS}, \cite {Harvey}, \cite{Matsuo}, \cite{Manin} and \cite{Yui} for inputs as above.

\section{Proof of Theorem \ref{thm:main2}: Algorithm for complete intersections}\label{sec:HWgeneral}

As in the previous sections, let $K$ be a perfect field of characteristic $p$, and let $S = K[X_0, \ldots , X_r]$ denote the polynomial ring of $r+1$ indeterminates over $K$.
Let $X \subset \mathbf{P}^r=\mathrm{Proj}(S)$ be an algebraic variety of dimension $q=\mathrm{dim}(X)$.

In this section, we give a specific method to compute the representation matrix for the Frobenius $F^{\ast,q}$ with $q = \mathrm{dim} (X)$ when $X=V(f_1, \ldots , f_t)$ is a {\it complete intersection} defined by an $S$-{\it regular} sequence $(f_1, \ldots , f_t) \in S^t$.
For this, we first prove that the representation matrix of $F^{\ast,q}$ for a suitable basis is given by coefficients in $(f_1 \cdots f_t)^{p-1}$.
Specifically in the case of $q=1$, this matrix is said to be the {\it Hasse-Witt matrix} of the curve $X$, which determines the superspecialty of $X$.

\begin{rem}
The author and Harashita already showed this method in \cite[Appendix B]{KH16}.
The author here re-writes the method as a special case of Algorithm (I) in Section \ref{sec:main_alg} of this paper. 
\end{rem}

First we collect some basic properties on regular sequences of modules.

\subsection{Regular sequences of modules}

\begin{defin}
Let $R$ be a commutative ring with unity, and $M$ an $R$-module.
An ordered $t$-tuple $( x_1, \ldots , x_t ) \in R^t$ is said to be an $M$-{\it regular sequence} (or simply $M$-{\it regular}) if
\begin{enumerate}
\item $M \big/ ( x_1, \ldots , x_t ) M \neq 0$, and
\item For each $1 \leq i \leq t$, the element $x_i$ is a {\it nonzerodivisor} in $M \big/ ( x_1, \ldots , x_{i-1} ) M$, i.e.,
there is no $x \in M \big/ ( x_1, \ldots , x_{i-1} ) M$ with $x \neq 0$ such that $x_i x = 0$.
\end{enumerate}
\end{defin}

\begin{lem}
\label{lem:regular1}
Let $R$ be a local commutative ring with unity, and $M$ an $R$-module.
If $\langle x_1, \ldots , x_t \rangle_R \subset R$ is a proper ideal containing an $M$-regular sequence of length $t$, then $( x_1, \ldots , x_t )$ is an $M$-regular sequence.
\end{lem}

\begin{lem}
\label{lem:regular2}
Let $R$ be a commutative ring with unity, and $M$ an $R$-module.
If $( x_1, \ldots , x_t ) \in R^t$ is $M$-regular, then $( x_1^n, \ldots , x_t^n )$ is $M$-regular for any $n > 0$.
\end{lem}

\begin{proof}
We show the statement by the induction on $t$.
Consider the case of $t = 1$.
Let $f  \in R$ be a polynomial such that $M \neq f M$ and $f$ is a nonzerodivisor in $M$.
Obviously we have $M \neq f^n M$.
Assume $f^n x = 0$ in $M$ for some $x \in M$.
Since $f$ is a nonzerodivisor, it follows that $f^{n-1} x$ equals $0 \in M$, and recursively $x = 0$.

Consider the case of $t > 1$.
Since $( f_1^n, \ldots , f_t^n ) M \subset ( f_1, \ldots , f_t ) M$, we have $( f_1^n, \ldots , f_t^n ) M \neq M$.
Here it suffices to show that $f_t$ is a nonzerodivisor in $M \big/ ( f_1^n, \ldots , f_{t-1}^n ) M$.
Indeed, if $f_t$ is a nonzerodivisor in $M \big/ ( f_1^n, \ldots , f_{t-1}^n ) M$ and
if $f_t^n x = 0$ in $M \big/ ( f_1^n, \ldots , f_{t-1}^n ) M$ for some $x \in M$, then $f_t^{n-1} x = 0$ in $M \big/ ( f_1^n, \ldots , f_{t-1}^n ) M$, and recursively $x = 0$ in $M \big/ ( f_1^n, \ldots , f_{t-1}^n ) M$.
Let $P$ be a prime ideal of $R$ with $\mathrm{Ann} (M) \subset P$.
We consider the localization
\[
\left( M \big/ ( f_1^n, \ldots , f_{t-1}^n ) M \right)_P \simeq M_P / ( f_1^n, \ldots , f_{t-1}^n ) M_P \quad \mbox{(as an $R_P$-module)}
\]
at $P$.
Note that if $f_t$ is a nonzerodivisor in $M_P \big/ ( f_1^n, \ldots , f_{t-1}^n ) M_P$, then $f_t$ is a nonzerodivisor in $M \big/ ( f_1^n, \ldots , f_{t-1}^n ) M$.
If there exists $1 \leq i \leq t$ such that $f_i \notin P$, then the either $M_P = ( f_1^n , \ldots , f_{t-1}^n ) M_P$ or $f_t \in ( R_P )^{\times}$, and thus the result holds.
From this, we may assume that $R$ is a local ring and that its maximal ideal contains $f_i$ for all $1 \leq i \leq t$.
The condition that $( f_1, \ldots , f_t )$ is $M$-regular implies that $( f_1, \ldots , f_{t-1}, f_t^n )$ is $M$-regular.
Applying Lemma \ref{lem:regular1}, it is concluded that $( f_t^n, f_1, \ldots , f_{t-1} ) $ is an $M$-regular sequence.
Consequently, repeating the argument, $( f_1^n, \ldots , f_t^n ) $ is an $M$-regular sequence.
\end{proof}



We next define the Koszul complex of {\it graded} free $S$-modules.

\begin{defin}[Koszul complex of graded free modules]
For homogeneous polynomials $f_1, \ldots , f_t \in S \smallsetminus \{ 0 \}$
and an index $i$, we define the following graded free $S$-module of rank $\binom{t}{i}$:
\[
K_i ( f_1, \ldots , f_t )_{\mathrm{grd}} := \bigoplus_{1 \leq j_1 < \cdots < j_i \leq t} S (-d_{j_1 \ldots j_i}) \mathbf{e}_{j_1 \ldots j_i},
\]
where we set $d_{j_1 \ldots j_i} := \sum_{k=1}^i \mathrm{deg} ( f_{j_k} )$.
We define the graded homomorphism $\varphi_i : K_i ( f_1, \ldots , f_t )_{\mathrm{grd}} \longrightarrow K_{i-1} ( f_1, \ldots , f_t )_{\mathrm{grd}}$ of degree zero by putting
\[
\varphi_i (\mathbf{e}_{j_1 \ldots j_i} ) := \sum_{k=1}^i (-1)^{k-1} f_{j_k} \mathbf{e}_{j_1 \ldots \hat{j_k} \ldots j_i}.
\]
The sequence $K (f_1, \ldots , f_t)_{\mathrm{grd}} :=( K_i ( f_1, \ldots , f_t)_{\mathrm{grd}}, \varphi_i )_i$ is a chain complex of graded free $S$-modules.
We call this complex the {\it graded Koszul complex defined by} $(f_1, \ldots , f_t)$.
It is straightforward that $K(f_1, \ldots , f_t)_{\mathrm{grd}}$ is exact if $( f_1, \ldots , f_t )$ is $S$-regular.
\end{defin}

\begin{lem}\label{lem:koszul}
With notation as above, let $(f_1, \ldots , f_t) \in S^t$ be an $S$-regular sequence with $\mathrm{gcd} (f_i, f_j) = 1$ for $i \neq j$.
To simplify the notation, we put 
\[
\begin{split}
M_i := K_i ( f_1, \ldots , f_t )_{\mathrm{grd}}, \quad \mbox{and} \ I := \langle f_1, \ldots , f_t \rangle_S, \\
M_i^{(n)} := K_i ( f_1^n, \ldots , f_t^n )_{\mathrm{grd}}, \quad \mbox{and} \ I_n := \langle f_1^n, \ldots , f_t^n \rangle_S.
\end{split}
\]
We denote by $\varphi_i^{(n)}$ the $i$-th differential of the complex $M_i^{(n)} = K ( f_1^n , \ldots , f_t^n)_{\mathrm{grd}}$.
We also define a graded homomorphism $\psi_i : K_i ( f_1^n, \ldots , f_t^n )_{\mathrm{grd}} \longrightarrow K_i ( f_1, \ldots , f_t)_{\mathrm{grd}}$ of degree zero as follows:
\[
\psi_i (\mathbf{e}_{j_1 \ldots j_i} ) := (f_{j_1} \cdots f_{j_i})^{n-1} \mathbf{e}_{j_1 \ldots j_i}.
\]
The following diagram of homomorphisms of graded $S$-modules commutes, and each horizontal sequence is exact:
$$\xymatrix{
0 \ar[r]^{\varphi_{t+1}^{( n )}} & M_t^{(n)}  \ar[d]^{\psi_{t}} \ar[r]^(0.6){\varphi_{t}^{( n )}} \ &  \cdots   \ar[r]^(0.5){\varphi_{2}^{( n )}} &             M_1^{(n)} \ar[d]^{\psi_{1}} \ar[r]^{\varphi_{1}^{( n )}} & M_0^{(n)} = S  \ar[d]^{\psi_{0}} \ar[r]^(0.5){\varphi_0^{( n )}} & M_{-1}^{(n)} := S / I_n                  \ar[d]^{\psi} \ar[r]    & 0            \\
0 \ar[r]^{\varphi_{t+1}^{( 1 )}} & M_t                                    \ar[r]^(0.6){\varphi_t^{( 1 )}} & \cdots                                   \ar[r]^(0.5){\varphi_2^{( 1 )}} & M_1  \ar[r]^{\varphi_{1}^{( 1 )}} & M_0 = S         \ar[r]^(0.5){\varphi_0^{( 1 )}} & M_{-1} := S / I     \ar[r]    & 0            
}$$
where $\psi_0$ is the identity map on $S$, and $\psi$ is the homomorphism defined by $h + I_n \mapsto h + I$.
\end{lem}

\begin{proof}
By our assumption together with Lemma \ref{lem:regular2}, the sequence $(f_1^n, \ldots , f_t^n)$ is $S$-regular, and hence the complex $K (f_1^n, \ldots , f_t^n)_{\mathrm{grd}} = (M_i^{(n)}, \varphi_i^{(n)})_i$ is exact.
We show that the diagram commutes.
For a basis element $\mathbf{e}_{j_1 \ldots j_{i+1}} \in M_{i+1}^{(n)}$, we have
\begin{eqnarray}
\left( \psi_i \circ \varphi_{i+1}^{(n)} \right) ( \mathbf{e}_{j_1 \ldots j_{i+1}} )
& = & \psi_i \left( \sum_{k=1}^{i+1} (-1)^{k-1} f_{j_k}^n \mathbf{e}_{j_1 \ldots \hat{j_k} \ldots j_{i+1}} \right) \nonumber \\
& = & \sum_{k=1}^{i+1} (-1)^{k-1} f_{j_k}^n \psi_i ( \mathbf{e}_{j_1 \ldots \hat{j_k} \ldots j_{i+1}} ) \nonumber \\
& = & \sum_{k=1}^{i+1} (-1)^{k-1} f_{j_k}^n ( f_{j_1} \cdots f_{j_{k-1}} f_{j_{k+1}} \cdots f_{j_{i+1}} )^{n-1} \mathbf{e}_{j_1 \ldots \hat{j_k} \ldots j_{i+1}} \nonumber \\
& = & (f_{j_1} \cdots f_{j_{i+1}})^{n-1} \sum_{k=1}^{i+1} (-1)^{k-1} f_{j_k} \mathbf{e}_{j_1 \ldots \hat{j_k} \ldots j_{i+1}}, \nonumber 
\end{eqnarray}
and
\begin{eqnarray}
\left( \varphi_{i+1}^{(1)} \circ \psi_{i+1} \right) ( \mathbf{e}_{j_1 \ldots j_{i+1}} )
& = & \varphi_{i+1}^{(1)} \left( (f_{j_1} \cdots f_{j_{i+1}})^{n-1} \mathbf{e}_{j_1 \ldots j_{i+1}} \right) \nonumber \\
& = & (f_{j_1} \cdots f_{j_{i+1}})^{n-1} \varphi_{i+1}^{(1)} \left( \mathbf{e}_{j_1 \ldots j_{i+1}} \right) \nonumber \\
& = & (f_{j_1} \cdots f_{j_{i+1}})^{n-1} \sum_{k=1}^{i+1} (-1)^{k-1} f_{j_k} \mathbf{e}_{j_1 \ldots \hat{j_k} \ldots j_{i+1}}. \nonumber 
\end{eqnarray}
Hence we have $\psi_i \circ \varphi_{i+1}^{(n)} = \varphi_{i+1}^{(1)} \circ \psi_{i+1}$.
\end{proof}

\subsection{The Frobenius action for complete intersections}

Here we give a specific method for computing the Frobenius $F^{\ast,q}: H^q (X, \mathcal{O}_X) \longrightarrow H^q (X, \mathcal{O}_X)$ with $q = \mathrm{dim}(X)$ when $X$ is a complete intersection.

\begin{prop}\label{prop:HW_more_general}
Let $K$ be a perfect field with $\mathrm{char} ( K )  = p > 0$.
Let $f_1, \ldots , f_{t}$ be homogeneous polynomials with $d_{j_1 \ldots j_{t-1}} \leq r$ for all $1 \leq j_1 < \cdots < j_{t-1} \leq t$ such that $\mathrm{gcd}( f_i, f_j ) = 1$ in $S:=K [ X_0, \ldots , X_r]$ for $i \neq j$. 
Suppose that the sequence $( f_1, \ldots , f_{t})$ is $S$-regular.
Let $X = V ( f_1, \ldots , f_{t} )$ be the variety defined by the equations $f_1 = 0, \ldots ,f_{t} = 0$ in $\mathbf{P}^r = \mathrm{Proj}(S)$, and $q:=\mathrm{dim}(X)=r-t$.
Write $(f_1 \cdots f_{t} )^{p-1} =  \sum c_{i_0, \ldots , i_{r}} X_0^{i_0} \cdots X_r^{i_r}$ and
\begin{equation}
\{ (k_0, \ldots , k_r) \in ( \mathbb{Z}_{<0} )^{r+1} : \sum_{i=0}^r k_i  = - \sum_{j=1}^{t} \mathrm{deg} ( f_j ) \} = \{ (k_0^{(1)}, \ldots , k_r^{(1)}), \ldots , (k_0^{(g)}, \ldots , k_r^{(g)} ) \}, \nonumber
\end{equation}
where we note that $g = \mathrm{dim}_K H^q (X, \mathcal{O}_X)$.
Then the representation matrix of the Frobenius $F^{\ast,}$ is given by
\begin{equation}
\left[
\begin{array}{ccc}
	c_{- k_0^{(1)} p + k_0^{(1)}, \ldots , - k_r^{(1)} p + k_r^{(1)}} & \cdots & c_{- k_0^{(g)} p + k_0^{(1)}, \ldots , - k_r^{(g)} p + k_r^{(1)}} \\
\vdots & & \vdots \\
	c_{- k_0^{(1)} p + k_0^{(g)}, \ldots , - k_r^{(1)} p + k_r^{(g)}} & \cdots & c_{- k_0^{(g)} p + k_0^{(g)}, \ldots , - k_r^{(g)} p + k_r^{(g)}}
\end{array}
\right]. \nonumber 
\end{equation}
\end{prop}

\begin{proof}
We use the same notation as in Lemma \ref{lem:koszul}, and take $n = p$.
Put $\varphi_i := \varphi_i^{(1)}$ and
\begin{equation}
\Phi_i:=\widetilde{\varphi_i},\quad \Phi_i^{(p)}:= \widetilde{\varphi_i^{(p)}},\quad \Psi:=\widetilde{\psi},\quad \mbox{and}\quad \Psi_i:=\widetilde{\psi_i}.
\end{equation}
By Lemma \ref{lem:koszul}, the following diagram commutes:
$$\xymatrix{
\qquad H^q (X, \mathcal{O}_X) \ar[r]^{\cong} \ar[d]^{\left(F_1 |_{X^{p}}\right)^{\ast,q}} \ar@(l,l)[dd]_{F^{\ast, q}} & H^{q+1} ( \mathbf{P}^r, \widetilde{I} ) \ar[r]^(0.45){\cong}  \ar[d]^{F_1^{\ast,q}} & \Ker \left( H^r( \Phi_{r-q}) \right) \ar[d]^{\mbox{\rm power } p} \ar[r]^(0.4){\cong} & H^r (\mathbf{P}^r, \mathcal{O}_{\mathbf{P}^r} (- \sum_{j=1}^t d_j)) \ar[d]^{\mbox{\rm power } p} \\
\qquad H^q (X^p, \mathcal{O}_{X^p}) \ar[r]^{\cong} \ar[d]^{H^q ( \Psi )} & H^{q+1} ( \mathbf{P}^r, \widetilde{I_p} ) \ar[r]^(0.45){\cong} \ar[d]^{H^{q+1} ( \Psi_0 )} &  \Ker \left( H^r ( \Phi_{r-q}^{(p)} ) \right) \ar[d]^{H^r (\Psi_{r-q})} \ar[r] & H^r (\mathbf{P}^r, \mathcal{O}_{\mathbf{P}^r} (- \sum_{j=1}^t d_j p)) \ar[d]^{(f_1 \cdots f_t)^{p-1}} \\
\qquad H^q (X, \mathcal{O}_X) \ar[r]^{\cong} & H^{q+1} ( \mathbf{P}^r, \widetilde{I} ) \ar[r]^(0.45){\cong} & \Ker \left( H^r \left( \Phi_{r-q} \right) \right) \ar[r]^(0.4){\cong} & H^r (\mathbf{P}^r, \mathcal{O}_{\mathbf{P}^r}(- \sum_{j=1}^t d_j))
}$$
where $F_1$ (resp. $F$) is the Frobenius morphism on $\mathbf{P}^r$ (resp. $X$) and $X^p :=  V \left( f_1^p, \ldots, f_{t}^p \right)$.
The $K$-vector space $H^r (\mathbf{P}^r, \mathcal{O}_{\mathbf{P}^r}(- \sum_{j=1}^t d_j))$ has a basis $\{ X_0^{k_0} \cdots X_r^{k_r} : (k_0, \ldots , k_r) \in ( \mathbb{Z}_{<0} )^{r+1} ,\ \sum_{i=0}^r k_i  = - \sum_{j=1}^{t} \mathrm{deg} ( f_j )  \}$.
For each $(k_0^{(i)}, \ldots , k_r^{(i)})$, we have
\begin{eqnarray}
(f_1 \cdots f_t)^{p-1} \cdot \left( X_0^{k_0^{(i)}} \cdots X_r^{k_r^{(i)}} \right)^p & = &(f_1 \cdots f_t)^{p-1} \cdot  X_0^{k_0^{(i)} p} \cdots X_r^{k_r^{(i)} p } \nonumber \\
& = & \sum c_{i_0, \ldots , i_{r}} X_0^{i_0+k_0^{(i)}p} \cdots X_r^{i_r+k_r^{(i)}p} \nonumber \\
& = & \sum_{j=1}^g c_{- k_0^{(i)} p + k_0^{(j)}, \ldots , - k_r^{(i)} p + k_r^{(j)}} X_0^{k_0^{(j)}} \cdots X_r^{k_r^{(j)}}. \nonumber
\end{eqnarray}
Hence our claim holds.
\end{proof}

\subsection{Algorithm and Complexity}

Proposition \ref{prop:HW_more_general} gives a simplification of Algorithm (I) in Section \ref{sec:main_alg} if the input $(f_1, \ldots , f_t)$ is $S$-regular and if $q = \mathrm{dim}(X)=r-t$:
To compute $F^{\ast,q}$, we not necessarily compute any free resolution, but only $(f_1 \cdots f_t)^{p-1}$.
Moreover this method is viewed as a generalization of a standard method to compute $F^{\ast,1}$ for elliptic curves, see Section \ref{sec:Intro} or \cite[Chapter I\hspace{-.1em}V]{Har}.
Here we write down an algorithm for complete intersections:

\paragraph{Algorithm (I\hspace{-.1em}I) (algorithm for complete intersections)}\label{sec:main_alg2}
Let $f_1, \ldots , f_{t}$ be homogeneous polynomials in $S = K [ x_0, \ldots , x_r ]$ with $d_{j_1 \ldots j_{t-1}} := \sum_{k=1}^{t-1} \mathrm{deg} ( f_{j_k} ) \leq r$ for all $1 \leq j_1 < \cdots < j_{t-1} \leq t$ such that $\mathrm{gcd}( f_i, f_j ) = 1$ in $S:=K [ X_0, \ldots , X_r]$ for $i \neq j$. 
Given $f_1, \ldots , f_t$ such that $(f_1, \ldots , f_t)$ is $S$-regular, a rational prime $p$ and an integer $q=r-t=\mathrm{dim}(X)$, we give an algorithm to compute the representation matrix for the action of Frobenius to $H^q ( X, {\cal O}_X )$, where $X = V ( f_1, \ldots , f_t )$.
Write $(f_1 \cdots f_{t} )^{p-1} =  \sum c_{i_0, \ldots , i_{r}} X_0^{i_0} \cdots X_r^{i_r}$ and
\begin{equation}
\{ (k_0, \ldots , k_r) \in ( \mathbb{Z}_{<0} )^{r+1} : \sum_{i=0}^r k_i  = - \sum_{j=1}^{t} \mathrm{deg} ( f_j ) \} = \{ (k_0^{(1)}, \ldots , k_r^{(1)}), \ldots , (k_0^{(g)}, \ldots , k_r^{(g)} ) \}, \nonumber
\end{equation}
where we set $g = \mathrm{dim}_K H^q (X, \mathcal{O}_X)$.
\begin{enumerate}
\item Compute the coefficients $c_{- k_0^{(i)} p + k_0^{(j)}, \ldots , - k_r^{(i)} p + k_r^{(j)}}$ for $1 \leq i, j \leq g$ in $(f_1 \cdots f_t)^{p-1}$.
\item Output
\begin{equation}
\left[
\begin{array}{ccc}
	c_{- k_0^{(1)} p + k_0^{(1)}, \ldots , - k_r^{(1)} p + k_r^{(1)}} & \cdots & c_{- k_0^{(g)} p + k_0^{(1)}, \ldots , - k_r^{(g)} p + k_r^{(1)}} \\
\vdots & & \vdots \\
	c_{- k_0^{(1)} p + k_0^{(g)}, \ldots , - k_r^{(1)} p + k_r^{(g)}} & \cdots & c_{- k_0^{(g)} p + k_0^{(g)}, \ldots , - k_r^{(g)} p + k_r^{(g)}}
\end{array}
\right]. \nonumber
\end{equation}
\end{enumerate}

\paragraph{Complexity}
The complexity heavily depends on one's choice of algorithms for computing the multiplication and the power computation over the multivariate polynomial ring $K [X_0, \ldots , X_r]$;
for a fixed $r$, the complexity can be bounded in polynomial time with respect to $\mathrm{max}_{1 \leq j \leq t}(\mathrm{deg} (f_j))$ and $p$, see e.g., \cite[Theorem 3.1]{Horo}.

\section{Examples and Experimental results}

This section shows computational examples and experimental results obtained by our implementation over Magma \cite{Magma}, \cite{MagmaHP}.

\subsection{Examples}





\begin{ex}\label{ex:X_67}
Let $K$ be a perfect field of characteristic $p > 2$.
Put
\begin{eqnarray}
f & := & 5 v z - 2 w x - 3 w y + w z, \nonumber \\
g & := & 10 v^2 + 5 w v - 5 w^2 + 4 x^2 - 12 x y + 2 x y - 2 y^2 - 35 y z - 12 z^2, \nonumber \\
h & := & 15 v^2 - 5 w v + 5 w^2 + 8 x^2 - 12 x y - 14 x z - 11 y^2 - 3 y z + 15 z^2, \nonumber 
\end{eqnarray}
and $C := V (f, g, h) \subset \mathbf{P}^4$.
The curve $C$ is the (classical) modular curve of level $67$, say $C = X_0 (67)$.
For defining equations for modular curves, see e.g., \cite{Gal_phd}.
In the following, we compute the representation matrix for the Frobenius action to $H^1 (C, \mathcal{O}_C)$ for the case of $p=3$.
In this case, we have the following commutative diagram:
$$\xymatrix{
0 \ar[r]^(0.45){\varphi_{4}^{( p )}} & S (- 6 p )  \ar[d]^{\psi_{3}} \ar[r]^(0.4){\varphi_{3}^{( p )}}  &  \bigoplus_{j=1}^3 S(-4 p )  \ar[d]^{\psi_{2}} \ar[r]^(0.5){\varphi_{2}^{( p )}} &              \bigoplus_{j=1}^3 S(-2 p ) \ar[d]^{\psi_{1}} \ar[r]^(0.7){\varphi_{1}^{( p )}} & S  \ar[d]^{\psi_{0}} \ar[r]^(0.5){\varphi_0^{( p )}} & S / I_p                  \ar[d]^{\psi} \ar[r]    & 0            \\
0 \ar[r]^(0.45){\varphi_{4}} & S (-6)                                    \ar[r]^(0.4){\varphi_3} & \bigoplus_{j=1}^3 S(-4 )                                  \ar[r]^(0.5){\varphi_2} & \bigoplus_{j=1}^3 S(-2 )   \ar[r]^(0.7){\varphi_{1}} & S         \ar[r]^(0.5){\varphi_0} & S / I     \ar[r]    & 0            
}$$
For the representation matrices of the above homomorphisms, see the text files on the web page of the author~\cite{HPkudo}.
The $1$st cohomology group $H^1 ( C, \mathcal{O}_C )$ has a basis
\[
\left\{ \frac{1}{x^2 y z v w}, \frac{1}{x y^2 z v w}, \frac{1}{x y z^2 v w}, \frac{1}{x y z v^2 w}, \frac{1}{x y z v w^2} \right\},
\]
which indicates that $C$ is a curve of genus $5$.
From the output of our program, the representation matrix of $F^{\ast}$ is
\[
\left[
\begin{array}{ccccc}
1 & 1 & 0 & 0 & 0 \\
2 & 0 & 2 & 0 & 0 \\
0 & 2 & 1 & 0 & 0 \\
0 & 0 & 0 & 0 & 0 \\
0 & 0 & 0 & 1 & 0
\end{array}
\right],
\]
and its rank is equal to $3$.
The Eigen polynomial is $a^5 + a^4 + a^3$, where $a$ is an indeterminate.
\end{ex}

\begin{rem}
In Example \ref{ex:X_67}, the tuple $(f, g , h )$ is $S$-regular with $\mathrm{deg} (f) =  \mathrm{deg}(g)= \mathrm{deg}(h)=2$.
Thus we can apply a method proposed in \cite[Appendix B]{KH16}.
\end{rem}

\begin{ex}\label{ex:23}
Let $K$ be a perfect field of characteristic $p > 0$.
Put
\begin{eqnarray}
f_1 & := & Y^2 + (- X_3 - X_1 - X_0 ) Y + 2 X_3 X_2 +3 X_1^2 - 2 X_1 X_0 + 2 X_0^2, \nonumber \\
f_2 & := & X_1^2 - X_0 X_2, \quad f_3 :=  X_2^2 - X_1 X_3, \quad f_4 := X_3 X_0 - X_2 X_1, \nonumber
\end{eqnarray}
and $C := V (f_1, f_2 , f_3, f_4) \subset \mathbf{P}^4$.
The curve $C$ is a normalization of the modular curve $X_0 (23)$, which is a hyperelliptic curve of genus $2$ given as an affine model in \cite{BN15}.
For a method of the normalization of hyperelliptic curves, see \cite[Chapter 10]{Gal}.
In what follows, we compute the representation matrix for the Frobenius action to $H^1 (C, \mathcal{O}_C)$ for the case of $p=5$.
By a similar way to Example \ref{ex:X_67}, we can compute a basis of $H^1 ( C, \mathcal{O}_C )$.
(For more information of the computation, see the text files on the web page of the author~\cite{HPkudo}.)
The output basis of $H^1 ( C, \mathcal{O}_C )$ is
\[
\left\{
\left[
\begin{array}{cccc}
0 & \frac{1}{X_0 X_1 X_2 X_3 Y} & 0  & 0
\end{array}
\right],
\left[
\begin{array}{cccc}
0 & 0 & \frac{1}{X_0 X_1 X_2 X_3 Y} &  0
\end{array}
\right]
\right\},
\]
which indicates that $C$ is a curve of genus $2$.
From the output of our program, the representation matrix of $F^{\ast}$ is
\[
\left[
\begin{array}{cc}
0 & 3\\
3 & 3
\end{array}
\right],
\]
and it has full-rank.
The Eigen polynomial is $a^2 + 2 a + 1$, where $a$ is an indeterminate.
\end{ex}

\subsection{Experimental results}

To confirm practical time performance of our algorithm,
we compute the representation matrix for the Frobenius action $F^{\ast,1}$ to the 1st cohomology group of $X_0 (23)$ for $3 \leq p \leq 17$.
Table \ref{table:ex} shows our experimental results for $X_0(23)$ of Example \ref{ex:23}.
We use the same notation as in Section \ref{sec:alg_comp}.

\begin{table}[htb]
 \begin{center}
   \caption{Experimental results to examine time performance on our algorithm for $X_0(23)$ in $\mathbf{P}^4$}\label{table:ex}
\vspace{0.2cm}
   \begin{tabular}{c|c|c|c|c|c|c} \hline
   $p$ & $\alpha$ & $D$ & Rank of $F^{\ast,1}$ & Eigen polynomial & Time for Step A & Time for Step B \\ \hline 
   $3$ &        $221$   & $7$ &                     $2$           &  $a^2 + 1$          &    ~~0.040                     & 0.010      \\ \hline
   $5$ &        $2975$   & $7$ &                     $2$           &  $a^2 + 2 a + 1$ &    ~~0.315                     & 0.015 \\ \hline
   $7$ &        $13720$   & $7$ &                     $2$           &  $a^2 + 5 a + 3$  &   ~1.680                    & 0.052 \\ \hline
   $11$ &      $104891$   & $7$ &                     $2$           &  $a^2 + 6 a + 4$  &   11.044                     & 0.360        \\ \hline
   $13$ &      $215664$   & $7$ &                     $2$           &  $a^2 + 7 a + 9$   &  24.186                     & 0.761        \\ \hline
   $17$ &      $676146$   & $7$ &                     $2$           &  $a^2 + 11 a + 4$ &  84.482                     & 2.664 \\ \hline   
	 \end{tabular}
\end{center}
\end{table}

\paragraph{Observation}
\begin{description}
\item[{\it Time performance:}] Recall from Corollary \ref{rem_comp3} that the binary complexity of Step B of Algorithm (I) is estimated as 
\begin{eqnarray}
O ( D^4 (\mathrm{log} (p))^3 + \alpha^2 D^2 (\mathrm{log} (p) )^4 ). 
\label{result_comp_5}
\end{eqnarray}
In the cases of our experiments, $D = 7$ is fixed. 
We here examine that time performance is better than the estimated upper bound $O ( \alpha^2(\mathrm{log} (p) )^4 )$. 
For $(p_1, \alpha_1) = (11,104891)$ and $(p_2, \alpha_2) = (17, 676146)$, we calculate the ratio $\alpha_2^2(\mathrm{log} (p_2) )^4 / \alpha_1^2(\mathrm{log} (p_1) )^4$.
One gets this value $\approx 131.33$.
From experimental results in Table 1, the ratio of time for Step B is $\approx 7.40$, which is smaller than $131.33$.
We can say the same for the other cases, e.g., $(p_3, \alpha_3) = (7, 13720)$ and $(p_4, \alpha_4) = (13, 215664)$.
From this, we observe that time performance of Step B is better than $O ( \alpha^2(\mathrm{log} (p) )^4 )$ for this computational example.
\item[{\it Correctness:}] 
In \cite{BN15}, an affine model of $X_0 (23)$ is given as
\[
y^2 + (- x^3 - x - 1) y = - 2 x^5 - 3 x^2 + 2 x - 2,
\]
ans its genus is $2$.
With Yui's method \cite{Yui}, one can calculate the rank and the Eigen polynomial of $F^{\ast,1}$,
which coincide with those in Table \ref{table:ex}.
\item[{\it Value of $\alpha$:}] For large $p$, we see that $\alpha$ takes an extremely large value.
However, rather than $p$, the value $\alpha$ depends on lifting morphisms computed in Step (A-2).
This means that a computational method adopted in Step (A-2) deeply affects total time performance and memory usage.
Thus, if one computes lifting morphisms such that $\alpha$ is small, Algorithm (I) may perform more efficiently, and save 
memory usage.
\end{description}

\section{Concluding remarks}

In this paper, we proposed an explicit algorithm to compute the representation matrix for the Frobenius action $F^{\ast,q}$ to the cohomology groups of  {\it arbitrary} varieties with defining equations in $\mathbf{P}^r$.
Under some assumptions, our algorithm was proved to be terminated in polynomial time with respect to three asymptotic parameters:
The first is the characteristic $p$, and the second is a certain mathematical invariant $D$ for the input variety.
The third is $\alpha$, which depends on lifting homomorphisms computed in a subroutine of our algorithm.
To confirm efficiency and the correctness, we also demonstrated some computational examples, one of which is $X_0(23)$, the (classical) modular curve of level $23$.
Our computational results coincide with theoretical results computed by Yui's method \cite{Yui} for defining equations as affine models \cite{BN15}.
This implies that our implementation over Magma returns in practical theoretically correct values.
Experimental results also suggest that time performance of our algorithm is better than an estimated upper bound.
From the experimental results together with our complexity analysis, we conclude that our algorithm provides a useful computational tool to investigate algebraic varieties and related invariants, which is expected to bring further theoretical/computational results in mathematics.




Our algorithm, however, performs efficiently under the assumption that one can get free resolutions and lifting homomorphisms efficiently.
Moreover, the value $\alpha$, which is determined by lifting homomorphisms, deeply affects time performance and memory usage.
Hence, we need to improve the computation of free resolutions and lifting homomorphisms, and find a new method to get lifting homomorphisms such that $\alpha$ takes a small value.
It is our future work.


\end{document}